\documentclass[11pt,a4paper]{article}

\usepackage{amsmath}
\usepackage{amsfonts}
\usepackage{amssymb}
\usepackage{multicol,color}
\usepackage{float}
\usepackage{theorem}
\usepackage{fancyhdr}
\usepackage{tcolorbox}
\usepackage{enumerate}
\usepackage[margin=1in]{geometry}

\usepackage{xcolor}

\usepackage[normalem]{ulem}

\usepackage{algorithm}
\usepackage{algpseudocode}

\usepackage{booktabs} 
\usepackage{array}    
\usepackage{comment}
\usepackage{subcaption}

\usepackage{setspace}
\usepgflibrary{shapes.geometric}

\def\disp{\displaystyle}

\def\hat{\widehat}\def\Hat{\widehat}
\def\Bar{\overline}
\def\la{\langle}
\def\ra{\rangle}

\def\R{\mathbb{R}}

\newtheorem{theorem}{Theorem}[section]
\newtheorem{lemma}[theorem]{Lemma}
\newtheorem{corollary}[theorem]{Corollary}
\newtheorem{proposition}[theorem]{Proposition}
\newtheorem{definition}[theorem]{Definition}
\newtheorem{assumption}[theorem]{Assumption}

\theoremstyle{plain}
\theorembodyfont{\rmfamily}

\begin{document}

\title{Extended SQP Methods in Nonsmooth Difference Programming  Applied to Problems with Variational Inequality Constraints}

\author{
  Boris S. Mordukhovich\thanks{Department of Mathematics \& Center for Artificial Intelligence and Data Science, Wayne State University, Detroit, MI 48202, USA. Email: \texttt{aa1086@wayne.edu}. Research of this author was partly supported by the US National Science Foundation under grant DMS-2204519 and by the Australian Research Council under Discovery Project DP250101112.} 
  \and
  Yixia Song\thanks{Department of Mathematics, Southern University of Science and Technology, Shenzhen, China. Email: \texttt{12131232@mail.sustech.edu.cn}.}
  \and
  Shangzhi Zeng\thanks{National Center for Applied Mathematics Shenzhen \& Department of Mathematics, Southern University of Science and Technology, Shenzhen, China. Email: \texttt{zengsz@sustech.edu.cn}. Research of this author was supported by the National Natural Science Foundation of China (12501429) and by the Shenzhen Fundamental Research Program (20250530150024003).}
  \and
  Jin Zhang\thanks{Corresponding author. \quad Department of Mathematics \& National Center for Applied Mathematics Shenzhen, Southern University of Science and Technology, Shenzhen, China. Email: \texttt{zhangj9@sustech.edu.cn}. Research of this author was supported by the National Key R\&D Program of China (2023YFA1011400).}
}

\date{}

\maketitle

\begin{abstract}
This paper explores a new class of constrained difference programming problems, where the objective and constraints are formulated as differences of functions, without requiring their convexity. To investigate such problems, novel variants of the extended sequential quadratic method are introduced. These algorithms iteratively solve strongly convex quadratic subproblems constructed via linear approximations of the given data by using their gradients and subgradients. The convergence of the proposed methods is rigorously analyzed by employing, in particular, the Polyak-\L ojasiewicz-Kurdyka property that ensures global convergence for various classes of functions in the problem formulation, e.g., semialgebraic ones. The original framework is further extended to address difference programming problems with variational inequality (VI) constraints. By reformulating VI constraints via regularized gap functions, such problems are naturally embedded into constrained difference programming that leads us to direct applications of the proposed algorithms. Numerical experiments for the class of continuous network design problems demonstrate the efficiency of the new methods.

\vspace{5pt}

\end{abstract}

\section{Introduction}\label{intro}

{\em Difference-of-convex} (DC) programming has emerged as a powerful framework for nonconvex and nonsmooth optimization owing to its ability of modeling complex structures while maintaining computational tractability; see, e.g., \cite{le2005,tao1997} and the references therein. In this work, we extend this paradigm by investigating a class of constrained {\em difference programming} problems formulated as follows:
\begin{equation}\label{difference-programming}
 \begin{aligned}
 &\min_{x \in X} \quad \varphi_0(x) := g_0(x) - h_0(x) \\
 &\text{subject to} \quad {\varphi_i(x) := g_i(x) - h_i(x) \leq 0, \quad i\in\mathcal{I},}
 \end{aligned}
 \end{equation}
where $\mathcal{I}$ is a finite set of indices, \( X \) is a nonempty, closed, and convex subset of \( \mathbb{R}^n \). The functions \( g_i: \mathbb{R}^n \to \mathbb{R} \) for \( i\in\{0\} \cup\mathcal{I}\)  are assumed to be continuously differentiable with locally Lipschitzian gradients (i.e., of class ${\cal C}^{1,1}$) on \( X \), while the functions \( h_i:\mathbb{R}^n \to\Bar\R:= \left(-\infty, \infty\right] \) for \( i \in \{0\} \cup \mathcal{I} \) are locally Lipschitz continuous and prox-regular over \( X \); see Section~\ref{sec:preliminaries}. Additionally, we assume that the objective function $\varphi_0$ is bounded from  below by a constant $m_{\varphi_0}$ over $X$. The primary objective of this study is to develop efficient numerical methods for solving \eqref{difference-programming}, establish a rigorous convergence analysis for the proposed algorithms, and provide valuable applications.

Recall that several classical algorithms exist to solve general constrained nonlinear programming (NLP). Among them, we mention sequential quadratic programming (SQP) methods, interior point methods, and augmented Lagrangian methods; see, e.g.,  \cite{bertsekas2016nonlinear,nocedal2006numerical} for more details.
The difference programming model \eqref{difference-programming} can be viewed as a generalization of DC programming, which has been a well-established topic in optimization with a variety of practical applications; see, in particular, \cite{thi2014dc,av20,ferreira,pang2017computing,ye2021difference} and the references therein among various developments on DC-type algorithms.

The unconstrained case of \eqref{difference-programming}, where $\varphi_i(x)$ for $i \in \mathcal{I}$ are absent, has been extensively investigated in \cite{aragon2023coderivative}, where the term ``difference programming" has been coined. In that work, the authors propose a novel Newton-type algorithm, RCSN, which effectively exploits second-order information via generalized differentiation techniques to solve such problems. In our work, we address the general setting of \eqref{difference-programming}, where the constraint functions $\varphi_i$ for $i\in\mathcal{I}$ are also expressed as differences of functions. The presence of such a structure in both objective and the constraint functions constitute additional analytical and computational challenges, which require the rigorous treatment of feasibility and stationarity conditions.

To develop an efficient numerical approach for solving \eqref{difference-programming}, we begin with considering SQP-type algorithms in the setting of smooth NLP problems. Methods of this type include solving a sequence of strongly convex quadratic subproblems to approximate the original one \cite{boggs1995sequential,lawrence2001}.
In particular, the {\em extended sequential quadratic method} (ESQM), introduced in \cite{auslender2013extended} for NLPs, serves as an inspiration for one of our algorithmic designs,  called CDP-ESQM, to solve \eqref{difference-programming}. Here we do not directly apply ESQM, but rather incorporate certain key ideas  from it to develop an iterative framework tailored for constrained difference programming \eqref{difference-programming}. Specifically, at each iteration we construct a strongly convex quadratic programming subproblem, where the problem's structural properties are incorporated through linear approximations derived from the gradients of \( g_i \) and limiting/Mordukhovich subgradients of \( h_i \) for \( i \in \{0\} \cup \mathcal{I} \). Under rather mild assumptions, we establish convergence properties of the proposed CDP-ESQM. {Moreover, under the assumption that a limiting subdifferential version of the Polyak-\L ojasiewicz-Kurdyka property holds for a suitably designed merit function, we establish global convergence results for CDP-ESQM. Notably, while CDP-ESQM utilizes a backtracking line search, our global convergence results do not rely on the assumption of a unit step size.}

Subsequently, we consider a class of {\em difference programs with variational inequality $(VI)$ constraints}.
A variational inequality problem seeks a vector $x^*$ such that a given mapping $F(x)$ satisfies
\begin{equation}\label{vi}
\langle F(x^*),x-x^* \rangle \geq 0\;\mbox{ for all }\;x\in \Omega,
\end{equation}
where $\Omega$ is a feasible set, which often models equilibrium problems in diverse fields of science and technology including machine learning, operations research, economics, optimization, etc.; see \cite{FacchineiPang2003,gidel2019,kinderlehrer2000variational,nagurney1993network}. VIs serve as a unifying tool to describe equilibrium conditions in applications to noncooperative games, market equilibrium models, contact mechanics, and generative adversarial networks. In various real-world models, decision making involves not only satisfying a VI constraint but also optimizing an objective function that leads us to problems of VI-constrained programming. Theoretical research on VI-constrained programming mostly focuses on the development of constraint qualifications and necessary optimality conditions as in \cite{ye2000constraint} and on the analysis of value function stability in parametric settings; see, e.g., \cite{DuttaLafhimZemkohoZhou2025,lucet2001sensitivity}. 
These studies, along with other early contributions, provide fundamental insights into the theoretical properties of VI-constrained programming.  In real-world applications, VI-constrained programming has been widely employed in various fields such as optimal shape design, Stackelberg-Cournot-Nash equilibria traffic equilibrium, etc.; see \cite{kocvara1992nondifferentiable,samadi2025improved}. Despite the goal of solving VI-constrained programming problems, the existing research has predominantly focused on a special case where the feasible set $\Omega$ in VI \eqref{vi} is a cone. Under this condition, problems of VI-constrained programming reduce to {\em mathematical programs with complementarity constraints} (MPCCs); see, e.g., \cite[Proposition~1.1.3]{FacchineiPang2003}. The reader may consult with \cite{guo2015solving,outrata,scholtes2001convergence,steffensen2010new} for various numerical algorithms developed to solve MPCCs.

Our aim in this part of the research is developing efficient algorithms to solve problems of difference programming with variational inequality constraints of type \eqref{vi}. To accomplish the goal, 
we build upon the single-level reformulation approach introduced in \cite{marcotte1996exact} and employ employ the {\em regularized gap function} for VI proposed in \cite{fukushima1992equivalent}.
This allows us to reformulate the VI constraint as an inequality constraint given by the difference of two (nonconvex) functions with reducing therefore the VI-constrained problem to the difference programming form \eqref{difference-programming}. Thus the proposed ESQM methods are directly applicable to the VI-constrained difference programming problems under consideration. The results established for such problems by using CDP-ESQM are applied to {\em continuous network design problems}. For this class of practical models, we conduct numerical experiments and compare the obtained numerical solutions with those computed by other well-recognized methods.\vspace*{0.03in}

The rest of the paper is organized as follows. 
Section~\ref{sec:preliminaries} presents the basic notation and needed preliminaries from variational analysis and nonsmooth optimization. In Section~\ref{sec:alg}, we design and provide convergence analysis of the proposed ESQM version for solving constrained difference programming problems (CDP-ESQM). Section~\ref{sec:plk} is devoted to proving the global convergence of CDP-ECQM, assuming a limiting subdifferential version of the Polyak-\L ojasiewicz-Kurdyka property holds for a specifically designed merit function, $E_{p}$ (defined subsequently in \eqref{eq: E definition}).
In Section~\ref{sec:VI-constrained}, we consider VI-constrained models of difference programming and provide adaptations of the CDP-ESQM algorithm to solve such problems by reducing them to the constrained difference programs \eqref{difference-programming} via the gap function regularization. Applying the developed algorithm to solving continuous network design problems  and numerical experiments are given in Section~\ref{sec:num}. The final Section~\ref{conc} summarizes the main achievements of the paper and discusses some topics of our future research.

\section{Preliminaries from Variational Analysis and Nonsmooth Optimization} \label{sec:preliminaries}
\setcounter{equation}{0}

Throughout this paper, we deal with finite-dimensional Euclidean spaces $\mathbb R^n$ and use the standard notation and terminology of variational analysis and generalized differentiation;  see, e.g., \cite{mordukhovich2018variational,rockafellar1998variational}. 
Let $\mathbb{B}_r(x)$ represent the closed ball centered at $x \in \mathbb{R}^n$
with radius $r > 0$, and let $\mathbb{B}_r$ denote the closed ball centered at the origin with the same radius, while the unit closed ball centered at the origin is denoted by $\mathbb{B}$.
The set of natural numbers is labeled as  $\mathbb{N} := \{1, 2, \dots\}$. 
The nonnegative and positive orthants in $\mathbb{R}^n$ are denoted by $\mathbb{R}^n_+$ and $\mathbb{R}^n_{++}$, respectively. We use the symbol $\langle \cdot, \cdot \rangle$ for the inner product and the symbol $\|\cdot\|$ for the Euclidean norm in $\mathbb{R}^n$. For a vector \( x \in \mathbb{R}^n \) and a closed set \( D \subset \mathbb{R}^n \), the distance from \( x \) to \( D \) is defined as \(\mathrm{dist}(x, D): = \min_{x' \in D} \|x - x'\|\). The indicator function of \( D \) is denoted by \(\delta_D\), and the normal cone associated with a convex set \( D \) at \( x\in D \) is given by 
\begin{equation}\label{nc-conv}
\mathcal{N}_D(x) := \big\{ v \in \mathbb{R}^n ~\big|~ \langle v, x' - x\rangle \le 0\;\mbox{  for all }\;x' \in D\big\}
\end{equation}
with the convention $\mathcal{N}_D(x) :=\varnothing$ used when $x \notin D$. The Euclidean projection onto \( D \) is signified by \(\mathrm{Proj}_D\) and the convex hull of $D$ by $\mathrm{co}\, D$. The Cartesian product of two sets \( X \) and \( Y \) is denoted by \( X \times Y \).

For any proper lower semicontinuous (l.s.c.) function  \( f: \mathbb{R}^n \to\Bar{\mathbb{R}}\), the associated domain and epigraph are given by $\text{dom}(f):=\left\{x\in\mathbb{R}^n ~|~ f(x)<\infty\right\}$ and $\text{epi}(f):=\left\{(x,\alpha)\in\mathbb{R}^n\times\mathbb{R} ~|~f(x)\leq\alpha\right\}$, respectively. The main subdifferential construction used in this paper is the (Mordukhovich) {\em limiting subdifferential} of $f$ at 
$\bar x\in{\rm dom}(f)$ defined as follows
\begin{equation}\label{sub}
\partial f(\bar{x}): = \left\{ v \in \mathbb{R}^n \;\middle|\; \exists x_k \xrightarrow{f} \bar{x}, \; \exists v_k \in \hat{\partial} f(x_k), \; v_k \to v \right\},
\end{equation}
where \( x_k \xrightarrow{f} \bar{x} \) indicates that \( x_k \to \bar{x} \) and \( f(x_k) \to f(\bar{x}) \) as \( k \to \infty \), and where $\hat{\partial} f(x)$ stands for the (Fr\'echet) {\em regular subdifferential} of $f$ at $x\in{\rm dom}\, f$ given by
\begin{equation}\label{rsub}
\hat\partial f(x):=\Big\{v\in\mathbb{R}^n\;\Big|\;\mathop{\lim\inf}\limits_{u\to x}\frac{f(u)-f(x)-\la v,u-x\ra}{\|u-x\|}\ge 0\Big\}.
\end{equation}
If \( \bar{x} \notin \mathrm{dom}(f)\), we put \(\partial f(\bar{x}): = \emptyset\). When $f$ is a convex function, both subdifferentials \eqref{sub} and \eqref{rsub} reduce to the classical subdifferential of convex analysis, but they can be very different even for simple nonconvex functions; see, e,g., $f(x):=-|x|$, $x\in\mathbb R$, where $\partial f(0)=\{-1,1\}$ while $\hat\partial f(0)=\emptyset$.  Note that, despite the nonconvex values of the limiting subdifferential, it enjoys a {\em full calculus} based on variational and extremal principles of variational analysis; see the books \cite{mordukhovich2006variational,mordukhovich2018variational,rockafellar1998variational} and the references therein.\vspace*{0.05in}

We say that a function $f\colon\mathbb{R}^n\to\Bar{\mathbb{R}}$ is {\em lower/subdifferentially regular} at $\bar x\in\rm{dom}\,f$ if we have $\partial f(\bar x)=\Hat\partial f(\bar x)$. As shown in the aforementioned books, various important classes of extended-real-valued functions possess this property. Many of such functions are prox-regular in the sense defined below.

\begin{definition}\label{definition prox-regular}
A function $f:\mathbb{R}^n\to\Bar{\mathbb{R}}$ is {\em prox-regular} at $\bar x \in \text{dom}(f)$ for $\bar v \in \partial f(\bar x)$ if it is l.s.c.\ around $\bar x$ and there exist $\varepsilon > 0$ and $r \geq 0$ such that
\begin{equation*}
\begin{aligned}
f(x') \geq f(x) + \langle v, x' - x \rangle - \frac{r}{2} \|x' - x\|^2
\end{aligned}
\end{equation*}
whenever $x, x' \in \mathbb{B}_\varepsilon(\bar x)$ with $f(x) \leq f(\bar x) + \varepsilon$ and $v \in \partial f(x) \cap \mathbb{B}_\varepsilon(\bar v)$. If the latter condition holds for all $\bar v \in \partial f(\bar x)$, $f$ is said to be prox-regular at $\bar x$. Given a set $X \subset \mathbb{R}^n$, the function $f$ is called  prox-regular on $X$ if it is prox-regular at every point of $X$.   
\end{definition}

The class of prox-regular functions has been extensively studied in modern variational analysis with a variety of applications. This class encompasses, in particular, convex and ${\cal C}^{1,1}$ functions, strongly amenable functions, lower-$\mathcal{C}^2$ functions, etc.; see \cite{rockafellar1998variational} for more details. It follows from \cite[Theorem 5.2]{Clarke1995} that any prox-regular and locally Lipschitzian function around $\bar x$, is subdifferentially regular at this point.\vspace*{0.05in}

Next we recall the notion of upper directional derivative for extended-real-valued functions.

\begin{definition}\label{definition upper directional derivative}
For a function $f:\mathbb{R}^n\to (-\infty,\infty ]$ and a point $\bar x\in\text{dom}(f)$, the {\em upper directional derivative} of $f$ at $\bar x$ in direction $\bar d\in\mathbb{R}^n$ is defined by
\begin{equation*}
\begin{aligned}
f'(\bar x;\bar d):=\limsup\limits_{\substack{t\to0^+\\d\to\bar d}}\frac{f(\bar x+td)-f(\bar x)}{t}.
\end{aligned}
\end{equation*}
\end{definition}
It is easy to see that if $f$ is locally Lipschitzian around $\bar x$, the expression for $f'(\bar x;\bar d)$ simplifies to
\begin{equation}
\begin{aligned}\label{dirder}
f'(\bar x;\bar d)=\limsup\limits_{t\to0^+}\frac{f(\bar x+t\bar d)-f(\bar x)}{t}.
\end{aligned}
\end{equation}

The following useful result provides a calculus rule for the upper directional derivative of maxima over finitely many locally Lipschitzian functions.

\begin{lemma}\label{max rule}
Let \( f_i: \mathbb{R}^n \to (-\infty,\infty ] \), \( i \in \mathcal{I} \), be a finite collection of proper functions, and let \( \bar{x} \in \bigcap_{i \in \mathcal{I}} \text{\rm dom}(f_i) \). Suppose that \( f_i \) is locally Lipschitzian around this point for all \( i \in \mathcal{I} \). Then the maximum function  
\( f = \max_{i \in \mathcal{I}} f_i \) is also locally Lipschitzian around \( \bar{x} \), and for any \( d \in \mathbb{R}^n \) we have 
\begin{equation}\label{max rule formula}
\begin{aligned}
f'(\bar{x}; d) = \max_{i \in \mathcal{I}_{\max}(\bar{x})} \left\{ (f_i)'(\bar{x}; d) \right\} \leq 
\max_{i \in \mathcal{I}} \left\{ f_i(\bar{x}) + (f_i)'(\bar{x}; d) \right\} - f(\bar{x}),
\end{aligned}
\end{equation}  
where \( \mathcal{I}_{\max}(\bar{x}) := \left\{ i \in \mathcal{I} \mid f_i(\bar{x}) = f(\bar{x}) \right\} \).  
In particular, it holds for \( \mathcal{I} = \{1, 2\} \) and \( f_2 \equiv 0 \) that
\begin{equation*}
\begin{aligned}
\left( \max\{f_1, 0\} \right)'(\bar{x}; d) \leq \max\{ f_1(\bar x) + (f_1)'(\bar{x}; d), 0 \} - \max\{ f_1(\bar x), 0 \}.
\end{aligned}
\end{equation*}
\end{lemma}

\begin{proof} \, It follows directly from the definition that \( f \) is locally Lipschitzian around \( \bar{x} \). Using the continuity of \( f_i \) at $\bar{x}$ and applying the same argument as in the proof of \cite[Theorem~3.24]{beck2017first} yield the existence of \( \varepsilon > 0 \) such that for any \( t \in (0, \varepsilon) \) we get the equality 
\begin{equation}\label{max1}
\begin{aligned}
\frac{f(\bar{x} + t d) - f(\bar{x})}{t} 
= \max_{i \in \mathcal{I}_{\max}(\bar{x})} \left\{\frac{f_i(\bar{x} + t d) - f_i(\bar{x})}{t} \right\}.
\end{aligned}
\end{equation}
Taking the limit as \( t \to 0^+ \) in \eqref{max1} brings us to the relationships
\begin{equation*}
\begin{aligned}
f'(\bar{x}; d) = & \limsup_{t \to 0^+} \frac{f(\bar{x} + t d) - f(\bar{x})}{t} \\
= & \limsup_{t \to 0^+} \max_{i \in \mathcal{I}_{\max}(\bar{x})} \left\{ \frac{f_i(\bar{x} + t d) - f_i(\bar{x})}{t} \right\} \\
= & \max_{i \in \mathcal{I}_{\max}(\bar{x})} \left\{ \limsup_{t \to 0^+} \frac{f_i(\bar{x} + t d) - f_i(\bar{x})}{t} \right\}\\
= &  \max_{i \in \mathcal{I}_{\max}(\bar{x})}  \left\{ (f_i)'(\bar{x}; d) \right\}.
\end{aligned}
\end{equation*}
Remembering that $f(\bar{x}) = f_i(\bar{x})$ for any $i \in \mathcal{I}_{\max}(\bar{x})$ tells us that
\begin{equation*}
\max_{i \in \mathcal{I}_{\max}(\bar{x})}  \left\{ (f_i)'(\bar{x}; d) \right\} =  \max_{i \in \mathcal{I}_{\max}(\bar{x})} \left\{ f_i(\bar{x}) + (f_i)'(\bar{x}; d) \right\} - f(\bar{x}) 
\leq  \max_{i \in \mathcal{I}} \left\{ f_i(\bar{x}) + (f_i)'(\bar{x}; d) \right\} - f(\bar{x}),
\end{equation*}
which verifies \eqref{max rule formula} and thus completes the proof of the lemma.
\end{proof}\vspace*{0.05in}

The next lemma, taken from \cite[Proposition 2.4]{aragon2023coderivative}, provides a relationship between the subdifferential \eqref{sub} of a locally Lipschitzian and prox-regular function $f$ and the upper directional derivative \eqref{dirder} of $-f$.

\begin{lemma}\label{minus prox regular}
Let $f:\mathbb{R}^n\to(-\infty,\infty]$ be locally Lipschitzian around $\bar x$ and prox-regular at this point. Then $f$ is lower regular at $\bar x$ with the property $\mathrm{co}\,\partial(-f)(\bar x) = -\partial f(\bar x)$ and the representation
\begin{equation*}
\begin{aligned}
(-f)'(\bar x,d) = \inf\left\{\langle -v,d \rangle \mid v \in \partial f(\bar x)\right\}\;\mbox{ for any }\;d\in\mathbb{R}^n.
\end{aligned}
\end{equation*}
\end{lemma}

We conclude this section with presenting the constraint qualification condition and stationarity notion for the difference programs \eqref{difference-programming} formulated via the limiting subdifferential and broadly used in the paper. The following constraint qualifications for \eqref{difference-programming}  and their terminology are inspired by \cite{xu2015smoothing, ye2000constraint}, where the reader can find more details and relationships with other constraint qualification conditions.

\begin{definition}\label{def:NNAMCQandENNAMCQ}
Let $\bar{x}$ be a feasible solution to problem (\ref{difference-programming}).
We say that the {\em no nonzero abnormal multiplier constraint qualification} (NNAMCQ) holds at $\bar{x}$ if for any $\lambda_i \geq 0$, $i \in \mathcal{I}$, not all equal to zero, that satisfy $\lambda_i \varphi_i(\bar{x}) = 0$ for all $i \in \mathcal{I}$, one has
\begin{equation}\label{NNAMCQ}
0\notin \sum_{i\in\mathcal{I}} \lambda_i \big(\nabla g_i(\bar x)-\partial h_i(\bar{x})\big)+\mathcal N_X(\bar x).
\end{equation}
The {\em extended no nonzero abnormal multiplier constraint qualification} (ENNAMCQ) holds at $\bar{x}\in X$ if for any $\lambda_i\geq0$, $i\in\mathcal{I}$, not all equal to zero, that satisfy $\lambda_i\varphi_i(\bar x) \geq 0$ for all $i\in\mathcal{I}$, one has condition (\ref{NNAMCQ}).
\end{definition}


\begin{definition}\label{def:critical-point}
Let $\bar{x}$ be a feasible solution to problem (\ref{difference-programming}). We say that $\bar{x}$ is a {\em stationary/KKT point} of (\ref{difference-programming}) if there exist multipliers $\lambda_i \geq 0$ for $i\in\mathcal{I}$ such that
\begin{equation}\label{kkt}
\begin{aligned}
&0 \in \nabla g_0(\bar{x}) - \partial h_0(\bar{x}) + \sum_{i\in\mathcal{I}} \lambda_i \big( \nabla g_i(\bar{x}) - \partial h_i(\bar{x}) \big) + \mathcal N_X(\bar{x}), \\
&\lambda_i \big( g_i(\bar{x}) - h_i(\bar{x}) \big) = 0, \quad  \forall i\in\mathcal{I}.
\end{aligned}
\end{equation}
\end{definition}

\section{ESQM for Constrained Difference Programming}\label{sec:alg}
\setcounter{equation}{0}
In this section, we propose a new algorithm of the ESQM type to solve constrained problems of difference programming  and conduct its convergence analysis. \vspace*{0.05in}

To design our version of ESQM for the difference program \eqref{difference-programming}, define the {\em penalized problem}
\begin{equation}\label{penalized_problem}
\min_{x \in X} \phi_p(x):=\frac{1}{p}\varphi_0(x)+ \sum_{i\in\mathcal{I}} \max\{\varphi_i(x),0\},
\end{equation}
where  \( p > 0 \) is the {\em penalty parameter}.  
At each iteration $k$, given the current iterate $x^k \in X$ and the penalty parameter $p_k$, we construct a linear approximation of $\phi_{p_k}(x)$ at $x^k$ by
\begin{equation}\label{eq:tilde_phi}
\tilde{ \phi}_{k}(x):= \frac{1}{p_k} \left( \varphi_0(x^k) + \langle \nabla g_0(x^k) - v_0^k, x - x^k\rangle \right) + \sum_{i\in\mathcal{I}} \max \left\{ \varphi_i(x^k) + \langle \nabla g_i(x^k) - v_i^k, x - x^k \rangle, 0 \right\}, 
\end{equation}
where $v_i^k$ is a limiting subgradient of $h_i$ at $x^k$, i.e., $v_i^k \in \partial h_i(x^k)$, for $i \in \mathcal{I}$. The update direction $d^k$ is obtained by solving the {\em strongly convex quadratic programming subproblem}:
\begin{equation}\label{direction_sub}
\begin{aligned}
\min_{d \in \mathbb{R}^n} \, \tilde{ \phi}_{k}(x^k +d) + {\frac{1}{2} \langle G_kd,d\rangle} + \delta_X(x^k + d),
\end{aligned}
\end{equation}
where $G_k \in \mathbb R^{n\times n}$ is a suitably chosen symmetric positive definite matrix designed to
flexibly incorporate second-order information of the smooth component \(g_0\) when available, thereby yielding a more accurate local quadratic model for the subproblem \eqref{direction_sub} and potentially improving the overall efficiency of the algorithm. We assume  $G_k$ satisfies the following uniform boundedness property throughout the paper.
\begin{assumption}\label{ass_G}
There exists $\alpha, M_G > 0 $ such that for all $k \ge 0$,
\[
\alpha\|d\|^2 \leq \langle G_k d, d \rangle \leq M_G \|d\|^2, \quad \forall d \in \mathbb{R}^n.   
\]
\end{assumption}
This assumption ensures the strong convexity of subproblem \eqref{direction_sub}, guaranteeing that it is computationally tractable and yields a unique optimal solution $d^k$. 

After determining the update direction $d^k$, we employ a {\em backtracking linesearch strategy} to compute an appropriate stepsize for updating the next iterate $x^{k+1}$. Specifically, set the stepsize $\tau_k: = \beta^{q_k}$, where $\beta \in (0,1)$ and $q_k$ is the smallest nonnegative integer ensuring the following condition
\begin{equation}\label{eq:linesearch}
{\phi_{p_k}(x^k + \tau_k d^k) \leq \tilde\phi_{k}(x^k+\tau_kd^k) + (1-\sigma) \alpha \tau_k \| d^k \|^2}
\end{equation}
where $\sigma \in (0,1)$ is a predefined parameter, and $\alpha$ is the positive definiteness modulus of $G_k$ introduced in Assumption \ref{ass_G}. {Condition \eqref{eq:linesearch} bounds the error between the penalized function and its local linearization. As we will establish later in Lemma \ref{lemma:sufficient_decrease}, this condition guarantees a sufficient decrease in $\phi_{p_k}(x)$
\begin{equation*}
    \begin{aligned}
        \phi_{p_k}(x^k + \tau_k d^k)\leq\phi_{p_k}(x^k)-\sigma\alpha\tau_k\|d^k\|^2.
    \end{aligned}
\end{equation*}
The existence of a finite step size $\tau_k$ satisfying this condition is guaranteed (see Proposition~\ref{lemma:stepsize_existence} below).} The new iterate $x^{k+1}$ is then updated by
\begin{equation*}
x^{k+1} = x^k + \tau_k d^k.
\end{equation*}
Note that the convexity of $X$ and the fact that $x^k,  x^k + d^k\in X$ ensure that $x^{k+1} \in X$.

At the end of each iteration $k$, we update the penalty parameter $p_k$.
The role of $p_k$ is to enforce the feasibility of the iterates (i.e., satisfying all constraints such that $\sum_{i\in\mathcal I} \max\{\varphi_i(x),0\} = 0$). We first compute an approximation of the constraint violation as
$$
t^k := \sum_{i\in\mathcal I} \max \{ \varphi_i(x^k) + \langle \nabla g_i(x^k) - v_i^k, d^k \rangle, 0 \}.
$$
Intuitively, a large violation $t^k$ suggests that $p_k$ is insufficient to enforce feasibility, necessitating an increase. However, for a given $p_k$, the current iterate $x^k$ is generally not an exact minimizer of the penalized problem \eqref{penalized_problem}. This inexactness complicates the assessment of whether $p_k$ is too small. We use the norm of the update direction, $\|d^k\|$, to measure this inexactness, since $d^k = 0$ corresponds to $x^k$ being a stationary point of \eqref{penalized_problem}. A smaller $\|d^k\|$ indicates that $x^k$ is close to a stationary point, meaning the constraint violation $t^k$ should be small. Conversely, a larger $\|d^k\|$ implies lower subproblem accuracy, and a larger $t^k$ can be temporarily tolerated.

Thus, we dynamically scale the threshold for $t^k$ relative to $\|d^k\|$. Specifically, if $t^{k}$ exceeds this threshold, i.e., $t^k \ge c_p \|d^k\|$ for a preset parameter $c_p > 0$, the penalty parameter $p_{k+1}$ is increased; otherwise, it remains unchanged. This relative thresholding is designed to increase $p_k$ when the penalization appears inadequate, helping to avoid excessively large penalty values in the early iterations.

With these ingredients in hand, we are ready to design the desired algorithm of the ESQM type to solve the class of constrained difference programming problems \eqref{difference-programming}. From now on, Algorithm~\ref{alg:ESQM} described below is referred to as {\em CDP-ESQM}.

\begin{algorithm}
\caption{ESQM for Solving Constrained Difference Programming (CDP-ESQM)}
\label{alg:ESQM}
\begin{algorithmic}[1]
\Require Initial iterate \( x^0 \in X \), linesearch parameters \( \beta \in (0,1) \) and \( \sigma \in (0,1) \),  penalty parameters  \( p_0, \varrho_p, c_p > 0 \).

\For{$k = 0, 1, \dots$}

\State Choose $v_i^k \in \partial h_i(x^k)$ for $i \in \{0\} \cup \mathcal{I}$.
    \State {Construct a symmetric positive definite matrix $G_k$ satisfying Assumption~\ref{ass_G}.}

\State Compute the update direction \( d^k \) by solving the strongly convex quadratic programming  subproblem \eqref{direction_sub}. 
\State \textbf{If} $d^k = 0$ and $\varphi_i(x^k) \le 0, \ \forall i \in \mathcal{I}$, \textbf{then} STOP and return $x^k$.

\State Set $\tau_k = 1$. 
\While{
{$\phi_{p_k}(x^k + \tau_k d^k) > \tilde\phi_k(x^k+\tau_kd^k) + (1-\sigma)\alpha\tau_k \| d^k \|^2$}
}
{$ \tau_k = \beta \tau_k$. }
\EndWhile

\State  Update the iterate: \( x^{k+1} := x^k + \tau_k d^k \).

\State Compute $t^k := \sum_{i\in\mathcal{I}} \max \{ \varphi_i(x^k) + \langle \nabla g_i(x^k) - v_i^k, d^k \rangle, 0 \}.$
\State Update the penalty parameter by
\[
p_{k+1} =
\begin{cases}
p_k + \varrho_p & \text{if } t^k \geq c_p \| d^k \|, \\
p_k & \text{otherwise.}
\end{cases}
\]

\EndFor
\end{algorithmic}
\end{algorithm}

In the rest of this section, we provide a detailed {\em convergence analysis} of CDP-ESQM in Algorithm~\ref{alg:ESQM}. Unless otherwise stated, our {\em standing assumptions} on the initial data of problem \eqref{difference-programming} are as formulated in Section~\ref{intro}, i.e., the set $X$ is closed and convex, the functions $g_i$ ($i \in \mathcal{I}$) are of class ${\cal C}^{1,1}$, and the functions $h_i$ ($i \in \mathcal{I}$) are prox-regular and Lipschitz continuous on $X$. The reader can observe that some of the results given below hold under more general assumptions. The main goal in what follows is to establish the {\em global subsequential convergence} of CDP-ESQM to a stationary point of \eqref{difference-programming} in the sense of Definition~\ref{def:critical-point}.

Remembering that $d^k$ is a solution to the convex subproblem \eqref{direction_sub} and applying the subdifferential calculus rules of convex analysis for finite sums and maxima (see, e,g., \cite[Theorems~3.18, 3.27]{mordukhovich2023easy}) and the first-order optimality conditions from \cite[Theorem~7.15]{mordukhovich2023easy}, we arrive at the following result.

\begin{lemma}\label{lem:dk_foc}
Let $d^k$ be the solution to subproblem \eqref{direction_sub}. Then there {exist $\lambda_{k,i} \in [0,1]$ for $i \in \mathcal{I}$} such that
\begin{equation}\label{eq:foc} 
0 \in \frac{1}{p_k}\left(\nabla g_0(x^k)-v_0^k\right) + {\sum_{i\in\mathcal{I}} \lambda_{k,i} \left(\nabla g_i(x^k)-v_i^k\right)} + {G_k} d^k+\mathcal N_X \left(x^k+d^k\right), 
\end{equation}
and, for each $i\in\mathcal{I}$, the following conditions hold
\begin{equation}\label{eq:multiplier}
\lambda_{k,i} \min\left\{ \varphi_i(x^k) + \langle \nabla g_i(x^k) - v_i^k, d^k \rangle, 0 \right\}=(1- \lambda_{k,i})\max\left\{ \varphi_i(x^k) +  \langle \nabla g_i(x^k) - v_i^k, d^k \rangle, 0 \right\} =0.
\end{equation} 
\end{lemma}\vspace*{0.05in}

The next useful lemma is derived by employing the equalities in \eqref{eq:multiplier}.

\begin{lemma}\label{lemma:multiplier_inequality}
Let $d^k$ solve subproblem \eqref{direction_sub}. Then for any {$\lambda_{k,i} \in \left[0, 1\right]$ ($i\in\mathcal{I}$)} satisfying \eqref{eq:multiplier}, we have, for each {$i \in \mathcal{I}$}
\begin{equation}\label{eq:multiplier_inequality}
\max\left\{\varphi_i(x^k) + \langle \nabla g_i(x^k) - v_i^k, d^k \rangle, 0\right\} - \max\left\{\varphi_i(x^k), 0\right\} \le  \lambda_{k,i} \langle \nabla g_i(x^k) - v_i^k, d^k \rangle.
\end{equation}
\end{lemma}
\begin{proof}\, {For each $i \in \mathcal{I}$,} picking any {$\lambda_{k,i} \in \left[0, 1\right]$} tells us that
\begin{equation*}
\lambda_{k,i} \varphi_i(x^k) \le \lambda_{k,i} \max\left\{\varphi_i(x^k), 0\right\} \le \max\left\{\varphi_i(x^k), 0\right\}.
\end{equation*}
Since {$\lambda_{k,i}$} satisfies  \eqref{eq:multiplier}, it follows that
\begin{equation*}
\max\left\{\varphi_i(x^k) + \langle \nabla g_i(x^k) - v_i^k, d^k \rangle, 0\right\} = \lambda_{k,i} \left(\varphi_i(x^k) + \langle \nabla g_i(x^k) - v_i^k, d^k \rangle\right).
\end{equation*}
Combining the two inequalities above yields
\begin{equation*}
\begin{aligned}
&\max\left\{\varphi_i(x^k) + \langle \nabla g_i(x^k) - v_i^k, d^k \rangle, 0\right\} 
- \max\left\{\varphi_i(x^k), 0\right\} \\
\le \, & \lambda_{k,i} \left(
\varphi_i(x^k) + \langle \nabla g_i(x^k) - v_i^k, d^k \rangle
\right) - \lambda_{k,i} \varphi_i(x^k) \\
= \, &\lambda_{k,i} \langle \nabla g_i(x^k) - v_i^k, d^k \rangle,
\end{aligned}
\end{equation*}  
which verifies \eqref{eq:multiplier_inequality} and thus completes the proof of the lemma.
\end{proof}\vspace*{0.05in}

Building upon the preceding lemmas, we now demonstrate the \emph{well-definedness} of the \emph{backtracking linesearch} procedure in the CDP-ESQM algorithm, ensuring that the stepsize $\tau_k$ can be determined through a finite number of backtracking steps.
\begin{proposition}
\label{lemma:stepsize_existence}
The backtracking linesearch procedure in CDP-ESQM terminates in a finite number of steps, i.e., there exists $q_k \in \mathbb{N}$ for which the stepsize $\tau_k = \beta^{q_k}$ satisfies \eqref{eq:linesearch}.
\end{proposition}
\begin{proof}\, 
Fix any given $k\in\mathbb N$. For each $i \in \mathcal{I}$, since {$g_i$} is of class $\mathcal{C}^{1,1}$ on $X$, Lemma A.11 in \cite{IzmailovSolodov2014} ensures that there exist constants {$r_{g_i}, \varepsilon_{g_i} > 0$} such that the inequality
\begin{equation*}
    \begin{aligned}
       g_i(x^k+ \tau d^k) \leq g_i(x^k) + \tau\langle\nabla g_i(x^k),d^k\rangle + \frac{r_{g_i}}{2}\tau^2\|d^k\|^2 
    \end{aligned}
\end{equation*}
holds for any $0 \le \tau \le \varepsilon_{g_i}$. Furthermore, for each $i \in \mathcal{I}$, the prox-regularity of {$h_i$} on $X$ guarantees the existence of constants {$r_{h_i}, \varepsilon_{h_i} > 0$} such that for any $0 \le \tau \le \varepsilon_{h_i}$, the following estimate holds:
\begin{equation*}
    \begin{aligned}
        h_i(x^k+\tau d^k) \geq h_i(x^k) + \tau\langle v_i^k,d^k \rangle - \frac{r_{h_i}}{2}\tau^2\|d^k\|^2. 
    \end{aligned}
\end{equation*}
Combining the above inequalities yields that for any $0 \le \tau \le \varepsilon' := \min_{i\in\mathcal{I}}\{\varepsilon_{g_i}, \varepsilon_{h_i}\}$, we have
\begin{equation*}
    \begin{aligned}
        \varphi_0(x^k+td^k) \leq \varphi_0(x^k) + \tau\langle \nabla g_0(x^k) -v_0^k,d^k \rangle + \frac{r_{g_0}+r_{h_0}}{2}\tau^2\|d^k\|^2
    \end{aligned}
\end{equation*}
and for each $i \in \mathcal{I}$,
\begin{equation*}
   \begin{aligned}
  \max\left\{ \varphi_i(x^k+ \tau d^k), 0 \right\} 
   \leq \, & \max\left\{ \varphi_i(x^k) + \tau\langle \nabla g_i(x^k) -v_i^k,d^k \rangle + \frac{r_{g_i}+r_{h_i}}{2}\tau^2\|d^k\|^2, 0 \right\} \\
   \leq \, & \max\left\{ \varphi_i(x^k) + \tau\langle \nabla g_i(x^k) -v_i^k,d^k \rangle, 0 \right\} + \frac{r_{g_i}+r_{h_i}}{2}\tau^2\|d^k\|^2.
   \end{aligned} 
\end{equation*}
Consequently, combining the above inequalities, it follows that
\begin{equation*}
    \begin{aligned}
       &\frac{1}{p_k}{\varphi_0}(x^k+ \tau d^k) + {\sum_{i\in\mathcal{I}}} \max\left\{ {\varphi_i}(x^k+\tau d^k), 0\right\} \\
       \leq \, & \frac{1}{p_k}\left( {\varphi_0}(x^k) + \tau \langle \nabla g_0(x^k)-v_0^k,d^k \rangle\right) + {\sum_{i\in\mathcal{I}}} \max\left\{ {\varphi_i}(x^k) + \tau \langle \nabla g_i(x^k)-v_i^k,d^k \rangle, 0 \right\} \\
       &\quad + \left( \frac{r_{g_0} + r_{h_0}}{2p_k} + \sum_{i\in\mathcal{I}} \frac{r_{g_i} + r_{h_i}}{2} \right) \tau^2 \|d^k\|^2
    \end{aligned}
\end{equation*}
holds for all $0 \le \tau \le \varepsilon_k'$. According to the definition of $\tilde{\phi}_k$ and by defining $r := (r_{g_0}+r_{h_0})/p_k + \sum_{i\in\mathcal I} (r_{g_i}+r_{h_i})$, we obtain for any $0 \le \tau \le \varepsilon' $,
\begin{equation}\label{eq:quadratic_bound}
      \phi_{p_k}(x^k+\tau d^k) \leq \tilde\phi_k(x^k+ \tau d^k)+\frac{r}{2}\tau ^2\|d^k\|^2.
\end{equation}
This further implies that when $0 \le \tau \le \varepsilon := \min\left\{\varepsilon', 2(1-\sigma)\alpha/r \right\}$, it holds that
\[
\phi_{p_k}(x^k+\tau d^k) \leq \tilde\phi_k(x^k+ \tau d^k)+ (1-\sigma) \alpha \tau \|d^k\|^2.
\]
This readily implies that
\begin{equation*}
\left\{\tau \mid 
\phi_{p_k}(x^k + \tau d^k) 
\leq \tilde\phi_k(x^k+ \tau d^k)+ (1-\sigma) \alpha \tau \|d^k\|^2, \,  \tau = \beta^q, q \in \mathbb{N} 	\right\} \neq
\varnothing,
\end{equation*}
which ensures that the backtracking linesearch terminates in a finite number of steps.
\end{proof}\vspace*{0.05in}

We next demonstrate the following key inequalies for deriving the sufficient decrease of penalized functions.

\begin{lemma}\label{lemma:descent_direction}
Let the direction $d^k$ solve subproblem \eqref{direction_sub}. Then there exist {$\lambda_{k,i} \in [0,1]$ for $i\in\mathcal{I}$} satisfying  conditions \eqref{eq:foc}, \eqref{eq:multiplier} and such that
\begin{equation}\label{eq:descent_direction}
(\phi_{p_k})'(x^k; d^k) \le 
\frac{1}{p_k} \left\langle \nabla g_0(x^k) - v_0^k, d^k \right\rangle 
+ \sum_{i\in\mathcal I} \lambda_{k,i} \left\langle\nabla g_i(x^k)-v_i^k, d^k\right\rangle \le -\langle G_kd^k,d^k \rangle.
\end{equation}
\end{lemma}

\begin{proof}\,
Since $d^k$ solves the convex subproblem \eqref{direction_sub}, Lemma~\ref{lem:dk_foc} provides the existence of $\lambda_{k,i} \in [0,1]$ ($i \in \mathcal{I}$) satisfying \eqref{eq:foc} and \eqref{eq:multiplier}. 
By definition of the upper directional derivative, we have
\begin{equation}\label{prop1_eq1}
  (\phi_{p_k})'(x^k; d^k) 
  \leq \frac{1}{p_k} \left( \langle \nabla g_0(x^k), d^k \rangle + (-h_0)'(x^k; d^k)\right)
  + \sum_{i\in\mathcal I} \left(\max\{\varphi_i, 0\}\right)'(x^k; d^k).
\end{equation}
Remembering that $h_0$ is locally Lipschitzian around $x^k$ and prox-regular at $x^k$, and that $v_0^k \in \partial h_0(x^k)$, we get by Lemma~\ref{minus prox regular} that
\begin{equation}\label{prop1_eq2}
 (-h_0)'(x^k; d^k) = \inf\left\{ \langle -v, d^k \rangle \mid v \in \partial h_0(x^k) \right\} \leq \langle -v_0^k, d^k \rangle.
\end{equation}
Applying now Lemma~\ref{max rule} to the terms $\left(\max\{\varphi_i, 0\}\right)'(x^k; d^k)$ tells us that for each $i \in \mathcal I$,
\begin{equation}\label{prop1_eq3}
 \left(\max\{{\varphi_i}, 0\}\right)'\left(x ^k; d^k\right) \le \max\left\{ {\varphi_i}(x^k) +{\varphi_i}'(x^k; d^k), 0 \right\} - \max\left\{ {\varphi_i}(x^k), 0 \right\}.
\end{equation}
Since $h_i$ is also locally Lipschitzian around $x^k$ and prox-regular at this point with $v_i^k \in \partial h_i(x^k)$, it follows from Lemma~\ref{minus prox regular} that
\begin{equation*}
 (-h_i)'(x^k; d^k)  \leq \langle -v_i^k, d^k \rangle.
\end{equation*}
Combining the latter with the smoothness of {$g_i$} yields for each $i \in \mathcal I$
 \begin{equation*}
   \varphi_i'(x^k; d^k) = \langle \nabla {g_i}(x^k), d^k \rangle + (-h_i)'(x^k; d^k) \le \langle \nabla g_i(x^k) - v_i^k, d^k \rangle,
 \end{equation*}
and therefore, summing over $i\in\mathcal I$, we deduce from \eqref{prop1_eq3} that
\begin{equation}\label{prop1_eq4}
 \begin{aligned}
   \sum_{i\in\mathcal I} \left(\max\{\varphi_i, 0\}\right)'\left(x^k; d^k\right) \leq\, &\sum_{i\in\mathcal I} \left( \max\left\{ \varphi_i(x^k) + \langle \nabla g_i(x^k) - v_i^k, d^k \rangle, 0 \right\} 
   - \max\left\{ \varphi_i(x^k), 0 \right\} \right) \\
   \le \, & \sum_{i\in\mathcal I} \lambda_{k,i} \left\langle \nabla g_i(x^k) - v_i^k, d^k \right\rangle,
 \end{aligned}
\end{equation}
while the second inequality therein follows from Lemma~\ref{lemma:multiplier_inequality}. Combining \eqref{prop1_eq1}, \eqref{prop1_eq2} and \eqref{prop1_eq4} yields the first estimate in \eqref{eq:descent_direction}. Next we use \eqref{eq:foc}, 
the fact that $x^k, x^k + d^k \in X$, and definition \eqref{nc-conv} of the normal cone $\mathcal{N}_X$ to obtain
\begin{equation*}
\left\langle \frac{1}{p_k}\left(\nabla g_0(x^k)-v_0^k\right) + \sum_{i\in\mathcal I} \lambda_{k,i} \left(\nabla g_i(x^k)-v_i^k\right) + G_kd^k, x^k + d^k - x^k \right \rangle \le 0,
\end{equation*}
which provides in turn the estimate
\begin{equation}\label{prop1_eq5}
\frac{1}{p_k}\left\langle\nabla g_0(x^k)-v_0^k, d^k\right\rangle + \sum_{i\in\mathcal I} \lambda_{k,i} \left\langle\nabla g_i(x^k)-v_i^k, d^k\right\rangle \le -\langle G_kd^k,d^k \rangle.
\end{equation}
Then \eqref{prop1_eq5} verifies the claimed inequalities in \eqref{eq:descent_direction}.
\end{proof}\vspace*{0.05in}

{
Now we show that the condition \eqref{eq:linesearch} in backtracking linesearch procedure for selecting an appropriate stepsize guarantees the updated iterate satisfy the sufficient decrease condition.
\begin{lemma}\label{lemma:sufficient_decrease}
Let $\{x^k\}$, $\{d^k\}$, and $\{\tau_k\}$ be the sequences generated by CDP-ESQM. Then we have
\begin{equation}\label{lem36_eq}
    \begin{aligned}
        \phi_{p_k}(x^{k+1})\leq\phi_{p_k}(x^k)-\sigma\alpha\tau_k\|d^k\|^2.
    \end{aligned}
\end{equation}
\end{lemma}}

\begin{proof}\,
By Lemma~\ref{lemma:descent_direction}, there exist multipliers $\lambda_{k,i}$ ($i \in \mathcal{I}$) satisfying conditions \eqref{eq:foc} and \eqref{eq:multiplier} such that
\begin{equation}\label{lem36_eq1}
    \frac{1}{p_k} \langle \nabla g_0(x^k) - v_0^k, d^k \rangle + \sum_{i \in \mathcal{I}} \lambda_{k,i} \langle\nabla g_i(x^k)-v_i^k, d^k\rangle \leq -\langle G_k d^k, d^k \rangle.
\end{equation}
Since these specific multipliers $\lambda_{k,i}$ satisfy \eqref{eq:multiplier} for each $i \in \mathcal{I}$, they fulfill the assumptions of Lemma~\ref{lemma:multiplier_inequality}. Furthermore, since $\tau_k \in (0, 1]$, we have the following convex combination:
\begin{equation*}
    \begin{aligned}
        \varphi_i(x^{k}) + \tau_k\langle\nabla g_i(x^k) - v_i^k, d^k\rangle = \tau_k\left(\varphi_i(x^{k}) + \langle\nabla g_i(x^k) - v_i^k, d^k\rangle\right) + (1-\tau_k)\varphi_i(x^{k}).
    \end{aligned}
\end{equation*}
Using the convexity of the function {$t \mapsto \max\{t,0\}$} yields, for each $i \in \mathcal{I}$,
\begin{equation*}
    \begin{aligned}
      &\max\left\{\varphi_i(x^{k}) + \tau_k\langle\nabla g_i(x^k) - v_i^k, d^k\rangle, 0\right\} \\ \leq \, & \tau_k\max\left\{\varphi_i(x^{k}) + \langle\nabla g_i(x^k) - v_i^k, d^k\rangle, 0\right\} + (1-\tau_k)\max\left\{\varphi_i(x^{k}), 0\right\}.  
    \end{aligned}
\end{equation*}
Next, applying Lemma~\ref{lemma:multiplier_inequality} to bound the first term on the right-hand side gives
\begin{equation*}
    \begin{aligned}&\tau_k\max\left\{\varphi_i(x^{k}) + \langle\nabla g_i(x^k) - v_i^k, d^k\rangle, 0\right\} + (1-\tau_k)\max\left\{\varphi_i(x^{k}), 0\right\} \\ \leq \,& \tau_k\left(\max\left\{\varphi_i(x^k), 0\right\} + \lambda_{k,i} \langle\nabla g_i(x^k) - v_i^k, d^k \rangle\right) + (1-\tau_k)\max\left\{\varphi_i(x^{k}), 0\right\} \\ = \,& \max\left\{\varphi_i(x^{k}), 0\right\} + \tau_k \lambda_{k,i} \left\langle \nabla g_i(x^k) - v_i^k, d^k \right\rangle.\end{aligned}
\end{equation*}

Consequently, the linear approximation $\tilde{\phi}_k$ evaluated at $x^k + \tau_k d^k$ satisfies
\begin{equation}\label{lem36_eq2}
    \begin{aligned}
        &\tilde{\phi}_k(x^k+\tau_kd^k)\\
        \leq \, & \frac{1}{p_k}\varphi_0(x^k) + \sum_{i \in \mathcal{I}} \max\{\varphi_i(x^k), 0\} + \frac{\tau_k}{p_k} \langle \nabla g_0(x^k) - v_0^k, d^k \rangle + \tau_k\sum_{i \in \mathcal{I}} \lambda_{k,i} \langle \nabla g_i(x^k) - v_i^k, d^k \rangle \\
        = \, & \phi_{p_k}(x^k) + \tau_k \left( \frac{1}{p_k} \langle \nabla g_0(x^k) - v_0^k, d^k \rangle + \sum_{i \in \mathcal{I}} \lambda_{k,i} \langle \nabla g_i(x^k) - v_i^k, d^k \rangle \right) \\
        \leq \, & \phi_{p_k}(x^k) - \tau_k\langle G_kd^k, d^k\rangle,
    \end{aligned}
\end{equation}
where the last inequality follows from \eqref{lem36_eq1}. Combining this result with the linesearch condition \eqref{eq:linesearch} and the Assumption \ref{ass_G} on $G_k$ that $\langle G_kd^k, d^k\rangle \geq \alpha\|d^k\|^2$ implies
\begin{equation*}
    \begin{aligned}
        \phi_{p_k}(x^{k+1}) = \phi_{p_k}(x^k+\tau_kd^k) &\leq \tilde{\phi}_k(x^k+\tau_kd^k) + (1-\sigma)\alpha\tau_k\|d^k\|^2 \\
        &\leq \phi_{p_k}(x^k) - \tau_k\langle G_kd^k, d^k\rangle + (1-\sigma)\alpha\tau_k\|d^k\|^2 \\
        &\leq \phi_{p_k}(x^k) - \sigma\alpha\tau_k\|d^k\|^2.
    \end{aligned}
\end{equation*}
This completes the proof.
\end{proof}

The result above justifies the well-definiteness of the backtracking linesearch procedure for selecting an appropriate stepsize and it guarantee that the sufficient decrease condition \eqref{lem36_eq} holds. This condition enables us in turn to verify a decrease property of the sequence
$$
A_k:=\frac{1}{p_k} \left( \varphi_0(x^k) - m_{\varphi_0} \right) +\sum_{i\in\mathcal I} \max \left\{ \varphi_i(x^k), 0 \right\},\quad k\in\mathbb{N}.
$$

\begin{proposition} \label{prop:bounded-subsequence}
Let $\{x^k\}$, $\{d^k\}$, and $\{\tau_k\}$ be the sequences generated by CDP-ESQM. Then we have
\begin{equation*}
\begin{aligned}
\frac{1}{p_{k+1}}\left(\varphi_0(x^{k+1}) - m_{\varphi_0}\right) +& \sum_{i\in\mathcal I} \max\left\{\varphi_i(x^{k+1}), 0\right\}\\
&\le   \frac{1}{p_k}\left(\varphi_0(x^{k}) - m_{\varphi_0}\right)+  \sum_{i\in\mathcal I} \max\left\{\varphi_i(x^{k}), 0\right\} - \sigma\alpha\tau_k\|d^k\|^2,
\end{aligned}
\end{equation*}
\begin{equation}\label{K}
\sum_{k \in \mathbb{N}} \| x^{k+1} - x^k \|^2 < \infty.
\end{equation}
\end{proposition}

\begin{proof}\,
It follows from Proposition~\ref{lemma:stepsize_existence} that the backtracking linesearch condition \eqref{eq:linesearch} is satisfied at each iteration. And then Lemma \ref{lemma:sufficient_decrease} gives us the inequalities
\begin{equation*}
\begin{aligned}
&\frac{1}{p_{k+1}}\left(\varphi_0(x^{k+1}) - m_{\varphi_0}\right) + \sum_{i\in\mathcal I} \max\{\varphi_i(x^{k+1}), 0\} \\
\leq \, &\frac{1}{p_k}\left(\varphi_0(x^{k+1}) - m_{\varphi_0}\right) + \sum_{i\in\mathcal I} \max\{\varphi_i(x^{k+1}), 0\} \\
\leq \, &\frac{1}{p_k}\left(\varphi_0(x^{k})- m_{\varphi_0}\right)+\sum_{i\in\mathcal I} \max\{\varphi_i(x^{k}),0\}
-\sigma\alpha\tau_k\|d^k\|^2,
\end{aligned}
\end{equation*}
where the first one is a consequence of the nonnegativity of $\varphi_0(x^{k+1}) - m_{\varphi_0}$ and the fact that $p_{k+1} \ge p_k$.

Next we investigate the series convergence in \eqref{K}. 
Summing up the above estimates with respect to $k$ and noting that $A_k \ge 0$ tell us that 
\begin{equation}\label{prop2_eq2}
\sum_{k = 1}^K \|x^{k+1} - x^k\|^2 =  \sum_{k = 1}^K\tau_k^2 \|d^k\|^2 \leq  \sum_{k = 1}^K \tau_k \|d^k\|^2 \leq \frac{1}{\sigma\alpha }  \sum_{k = 1}^K \left(A_k - A_{k+1}\right) \leq \frac{A_1}{\sigma\alpha}.
\end{equation}
for any $K\in\mathbb{N}$. This readily verifies \eqref{K} and completes the proof of the proposition.
\end{proof}\vspace*{0.07in}

To establish the convergence properties of CDP-ESQM based on the {decrease of $\{A_k\}$}, we need to further demonstrate that the stepsize {sequence $\{\tau_k\}$} is bounded away from zero. 

\begin{proposition} \label{prop:tau-lower-bound}
Let $\left\{x^k\right\}$, $\left\{d^k\right\}$, and $\left\{\tau_k\right\}$ be the sequences generated by CDP-ESQM.
For any bounded subsequence $\left\{x^{k_j}\right\}$ of $\left\{x^k\right\}$, we have that $\inf_{j \in \mathbb{N}} \tau_{k_j} > 0$ and $d^{k_j} \rightarrow 0$ as $j\to\infty$.
\end{proposition}

\begin{proof}\,
Assume on the contrary that there exists a subsequence $\{\tau_{k_l}\}$ of $\{\tau_{k_j}\}$ such that $\tau_{k_l} \rightarrow 0$ as $l \rightarrow \infty$. Without loss of generality, suppose that $\tau_{k_l} < 1$ for all $l$.
The backtracking linesearch strategy yields 
\begin{equation}\label{violate linesearch}
\phi_{p_{k_l}}\left(x^{k_l} + \beta^{-1}\tau_{k_l} d^{k_l}\right) > \tilde\phi_{k_l}\left(x^{k_l} + \beta^{-1}\tau_{k_l} d^{k_l}\right) +  (1-\sigma)\alpha\beta^{-1}\tau_{k_l}\|d^{k_l}\|^2\;\mbox{ for all }\;l\in\mathbb N.
\end{equation}

Since the sequence $\{x^{k_l}\}$ is bounded, we may assume that it converges to some point $\bar{x}\in X$. It follows from the continuous differentiability of the functions $g_i$ on $X$ that $\nabla g_i(x^{k_l}) \to \nabla g_i(\bar{x})$ for $i \in \{0\} \cup \mathcal{I}$. The local Lipschitz continuity of $h_i$ on $X$ and the subgradient inclusions $v_i^{k_l} \in \partial h_i(x^{k_l})$ ensure that $h_i(x^{k_l}) \to h_i(\bar{x})$ and that the sequences $\{v_i^{k_l}\}$ are bounded for all $i \in \{0\} \cup \mathcal{I}$. By passing to a subsequence if necessary, we obtain $v_i^{k_l} \to \bar{v}_i$ as $l \to \infty$ for all $i \in \{0\} \cup \mathcal{I}$. Due to the outer semicontinuity of the limiting subdifferential mappings $\partial h_i$, we then have $\bar{v}_i \in \partial h_i(\bar{x})$ for $i\in
\{0\}\cup\mathcal I$.

It follows from Proposition~\ref{prop:bounded-subsequence} that $\tau_k \|d^k\| = \|x^{k+1} - x^k\| \to 0$, which implies in turn that $x^{k_l} + \beta^{-1}{\tau_{k_l}} d^{k_l} \rightarrow \bar{x}$ as $l \rightarrow \infty$. 
Since  \( g_i \) for \( i\in\{0\} \cup\mathcal{I}\) are of class ${\cal C}^{1,1}$ on \( X \), there are \( r_{g_i} > 0 \) for \( i\in\{0\} \cup\mathcal{I}\) such that
\begin{equation}
g_i\left(x^{k_l} + \beta^{-1} \tau_{k_l} d^{k_l}\right) \leq g_i(x^{k_l}) + \beta^{-1} \tau_{k_l} \langle \nabla g_i(x^{k_l}), d^{k_l} \rangle + \frac{r_{g_i}}{2} \beta^{-2} \tau_{k_l}^2 \|d^{k_l}\|^2, \quad i\in\{0\} \cup\mathcal{I},  \label{g_descent} 
\end{equation}
for all $l\in\mathbb N$ sufficiently large; see, e.g., \cite[Lemma A.11]{IzmailovSolodov2014}). Furthermore, the prox-regularity of \( h_i \) at $\bar{x}$ gives us constants \( r_{h_i} > 0 \) for \( i\in\{0\} \cup\mathcal{I}\) ensuring the fulfillment of the estimates
\begin{equation}
h_i\left(x^{k_l} + \beta^{-1} \tau_{k_l} d^{k_l}\right) \ge  h_i(x^{k_l}) + \beta^{-1} \tau_{k_l} \langle v_i^{k_l}, d^{k_l} \rangle - \frac{r_{h_i}}{2} \beta^{-2} \tau_{k_l}^2 \|d^{k_l}\|^2\label{h_descent} 
\end{equation}
for \( i\in\{0\} \cup\mathcal{I}\) and all large $l\in\mathbb N$. 
Combining \eqref{g_descent} and \eqref{h_descent} yields
\begin{equation}\label{gh_descent}
\begin{aligned}
&  g_i\left(x^{k_l} + \beta^{-1} \tau_{k_l} d^{k_l}\right) - h_i\left(x^{k_l} + \beta^{-1} \tau_{k_l} d^{k_l}\right)  \\ 
\leq \, &g_i\left(x^{k_l}\right) - h_i\left(x^{k_l}\right) + \beta^{-1} \tau_{k_l} \langle \nabla g_i(x^{k_l}) - v_i^{k_l}, d^{k_l} \rangle + \frac{r_{g_i} + r_{h_i}}{2} \beta^{-2} \tau_{k_l}^2 \|d^{k_l}\|^2 
\end{aligned}
\end{equation}
for \( i\in\{0\} \cup\mathcal{I}\) and all large $l\in\mathbb N$. 
Then using \eqref{gh_descent} and an argument similar to that in the proof of Proposition~\ref{lemma:stepsize_existence} for deriving
\eqref{eq:quadratic_bound}, we can show that
\begin{equation}\label{merit function decent}
\begin{aligned}
\phi_{p_{k_l}} \left(x^{k_l} + \beta^{-1} \tau_{k_l} d^{k_l}\right) \leq \tilde\phi_{k_l} \left(x^{k_l} + \beta^{-1} \tau_{k_l} d^{k_l}\right)+\left(\frac{r_{g_0} + r_{h_0}}{2p_{k_l}} + \sum_{i\in\mathcal I} \frac{r_{g_i} + r_{h_i}}{2}\right)\beta^{-2}\tau_{k_l}^2 \|d^{k_l}\|^2.
\end{aligned}
\end{equation}
Combining \eqref{violate linesearch} and \eqref{merit function decent} and dividing by $\beta^{-1}\tau_{k_l}\|d^{k_l}\|^2$ yields 
\begin{equation*}
(1-\sigma)\alpha < \left(\frac{r_{g_0} + r_{h_0}}{2p_{k_l}} + \sum_{i\in\mathcal I} \frac{r_{g_i} + r_{h_i}}{2}\right)\beta^{-1}\tau_{k_l}
\end{equation*}
for large $l$. Passing to the limits as $l\to\infty$ in the above inequality and noting that $\tau_{k_l} \rightarrow 0$, we obtain $(1-\sigma)\alpha \leq 0$, which contradicts the fact that $0<\sigma < 1$ and $\alpha>0$. Therefore, $\inf_{j \in \mathbb{N}} \tau_{k_j} > 0$, and then the claimed convergence $d^{k_j} \rightarrow 0$ follows from $\tau_{k_j}\|d^{k_j}\| = \|x^{k_j+1} - x^{k_j} \| \rightarrow 0$ as established in Proposition~\ref{prop:bounded-subsequence}.
\end{proof}\vspace*{0.07in}

Based on the above propositions, we now establish the convergence theorem for CDP-ESQM.

\begin{theorem}\label{general problem convergence}
Let \( \{x^k\} \) and $\{p_k\}$ be the sequences generated by the algorithm CDP-ESQM. 
Then this algorithm either terminates at a stationary point of the difference programming problem {\rm(\ref{difference-programming})}, or for any subsequence $\{x^{k_j}\}$ with $x^{k_j} \rightarrow \bar{x}$ as $j \rightarrow \infty$ we have the following assertions:
\begin{enumerate}  
\item[\bf(i)] If $\{p_k\}$ is bounded, then \( \bar{x} \) is a stationary point of {\rm(\ref{difference-programming})}.
    
\item[\bf(ii)] If $\{p_k\}$ is unbounded, then there exist $\bar{\lambda}_i \in [0,1]$, for $i\in\mathcal{I}$ such that $\bar{\lambda}_i =0$ whenever $\varphi_i(\bar x)<0$, $\bar{\lambda}_i =1$ whenever $\varphi_i(\bar x)>0$, and
\begin{equation*}
    \begin{aligned}
        0\in\sum_{i\in\mathcal I}\lambda_i\left(\nabla g_i(\bar x)-\partial h_i(\bar x)\right)+\mathcal{N}_X(\bar x).
    \end{aligned}
\end{equation*}
\end{enumerate}  
Moreover, if the sequence $\{x^k\}$ is bounded, then the set of its accumulation points is nonempty, closed, and connected. If finally the sequence \( \{x^k\} \) has an isolated accumulation point \( \bar x \), then the entire sequence \( \{x^k\} \) converges to \( \bar x \) as $k\to\infty$.
\end{theorem}

\begin{proof} \, Suppose first that CDP-ESQM terminates in finite steps, and thus there exists an index $K\in\mathbb N$ such that $d^K = 0$ and $\varphi_i(x^K) \le 0 \text{ for all } i\in\mathcal{I}$. Since $x^K \in X$, this point is feasible to the constrained difference programming (\ref{difference-programming}). Furthermore, since $d^K = 0$ solves the convex subproblem \eqref{direction_sub}, Lemma~\ref{lem:dk_foc} tells us that there exist $\lambda_{k,i} \in [0,1]$ ($i\in\mathcal I$) satisfying \eqref{eq:foc} and \eqref{eq:multiplier}. Consequently, we have the conditions
\begin{equation*}
\begin{aligned}
0 \in \nabla g_0(x^K)-v_0^K + p_K \sum_{i\in\mathcal I} \lambda_{k,i} \left(\nabla g_i(x^K)-v_i^K\right) + \mathcal N_X \left(x^K\right), \quad \lambda_{k,i} \varphi_i(x^K) = 0, \quad i\in\mathcal I.
\end{aligned}
\end{equation*}
Taking into account that $v_i^K \in \partial h_i(x^K)$ for $i \in \{0\} \cup \mathcal I$ guarantees that $x^K$ is a stationary point of the constrained difference programming problem (\ref{difference-programming}).

Consider next a subsequence $\{x^{k_j}\}$ such that $x^{k_j} \rightarrow \bar{x}$ as $j \rightarrow \infty$ with $\bar{x} \in X$. Following the same reasoning as in the proof of Proposition~\ref{prop:tau-lower-bound} and remembering that the functions $g_i \ (i\in\{0\}\cup \mathcal I)$ are continuously differentiable on $X$ ensure that $\nabla g_i(x^{k_j}) \to \nabla g_i(\bar{x})$ as $j\to\infty$ for $i \in \{0\} \cup \mathcal I$. It follows from the local Lipschitz continuity of $h_i \ (i\in\{0\}\cup\mathcal I)$ on $X$ that $h_i(x^{k_j}) \rightarrow h_i(\bar{x})$, and that the sequences $\{v_i^{k_j}\}$ are bounded for $i\in\{0\}\cup\mathcal I$. Without loss of generality, suppose that $v_i^{k_j} \rightarrow \bar{v}_i \in \partial h_i(\bar{x})$ as $j \to \infty$ for $i \in \{0\} \cup \mathcal I$. 

By Lemma~\ref{lem:dk_foc}, there {exist $\lambda_{k_j,i} \in [0, 1]$ ($i\in\mathcal I$)} satisfying \eqref{eq:foc} and \eqref{eq:multiplier} at the $k_j$-th iteration. Since ${\{\lambda_{k_j,i}\}}$ are bounded, we can assume that ${\lambda_{k_j,i} \to \bar{\lambda}_i \in [0, 1] \text{ for each } i\in\mathcal I}$. 
Furthermore, Proposition~\ref{prop:tau-lower-bound} ensures that $d^{k_j} \rightarrow 0$ as $j \rightarrow \infty$. If $\{p_k\}$ is bounded, then there exists a natural number $k_0 > 0$ such that $p_k = p_{k_0}$ for all $k \ge k_0$. The penalty update strategy in CDP-ESQM tells us that for all $k_j \ge k_0$ we have
\begin{equation}
c_p \|d^{k_j}\| > t^{k_j} = {\sum_{i\in\mathcal I} \max \left\{ \varphi_i( x^{k_j}) + \langle \nabla g_i(x^{k_j}) - v_i^{k_j}, d^{k_j} \rangle, 0 \right\}}.
\end{equation}
Passing to the limit as $j \to \infty$ in the above inequality and noting that $d^{k_j} \rightarrow 0$ give us ${\sum_{i\in\mathcal I} \max\{\varphi_i(\bar{x}), 0\} \le 0}$, which verifies that ${\varphi_i(\bar{x}) \le 0 \text{ for all } i\in\mathcal I}$, meaning $\bar{x}$ is a feasible solution to the constrained difference programming problem (\ref{difference-programming}). Taking $k = k_j$ and passing to the limit as $j \to \infty$ in \eqref{eq:foc} and \eqref{eq:multiplier} lead us to
\begin{equation*}
0 \in \nabla g_0(\bar{x}) - \bar{v}_0 
+ p_{k_0} {\sum_{i\in\mathcal I} \bar{\lambda}_i \left( \nabla g_i(\bar{x}) - \bar{v}_i \right)} + \mathcal N_X(\bar{x}),  \quad {\bar{\lambda}_i\varphi_i(\bar{x}) = 0, \quad i\in\mathcal I.}
\end{equation*}
Since $\bar{v}_i \in \partial h_i(\bar{x})$ for ${i \in \{0\} \cup  \mathcal I}$, we conclude that $\bar{x}$ is a stationary point of the constrained difference programming problem (\ref{difference-programming}).

On the other hand, if $\{p_k\}$ is unbounded, then by taking $k = k_j$ and passing to the limit as \( j \to \infty \) in \eqref{eq:foc} and \eqref{eq:multiplier}, and noting that $d^{k_j} \rightarrow 0$ while $G_k$ is bounded by Assumption \ref{ass_G}, we get the existence of $\bar{\lambda}_i \in [0,1]$, for $i\in\mathcal{I}$ satisfying
\[
 0 \in \sum_{i\in\mathcal{I}} \bar{\lambda}_i \left( \nabla g_i(\bar{x}) - \bar{v}_i \right) + \mathcal N_X(\bar{x}),
\]
and 
\[
\bar{\lambda}_i \min\left\{ \varphi_i(\bar{x}), 0 \right\}=(1- \bar{\lambda}_i)\max\left\{ \varphi_i(\bar{x}) , 0 \right\} =0.
\]

Finally, Proposition~\ref{prop:bounded-subsequence} tells us that $\lim_{k \to \infty}\|x^{k+1} - x^k\| = 0$, which is the classical Ostrowski convergence condition. Thus the remaining results follow directly from \cite[Theorem~8.3.9 and Proposition~8.3.10]{FacchineiPang2003}, and we are done with the proof of the theorem.
\end{proof}\vspace*{0.05in}

Theorem~\ref{general problem convergence} verifies, in particular, that the boundedness of the penalty parameter $p_k$ ensures that any accumulation point of the sequence $\{x^k\}$ is a stationary point. The following proposition shows that the sequence of penalty parameters \( \{p_k\} \) remains bounded under the ENNAMCQ condition.

\begin{proposition}\label{boundedness of penalty parameter sequence}
Suppose that CDP-ESQM does not terminate, and  let \( \{x^k\} \) and $\{p_k\}$ be the sequences generated by this algorithm. If ENNAMCQ holds at any accumulation points of $\{x^{k}\}$, then the sequence of penalty parameters  $\{p_k\}$ is bounded.
\end{proposition}

\begin{proof}\,
Assume on the contrary that $\{p_k\}$ is unbounded. By the penalty update rule in CDP-ESQM, there exists a subsequence $\{x^{k_j}\}$ such that $t^{k_j} \ge c_p\|d^{k_j}\| > 0$. Since $\{x^{k_j}\}$ is bounded, we can assume without loss of generality that $x^{k_j}$ converges to some $\bar{x} \in X$. Because $t^{k_j} > 0$, it follows that $\sum_{i\in\mathcal{I}}\max\{\varphi_i(\bar{x}),0\} \ge 0$. Otherwise, $\varphi_i(\bar{x}) <0 $ for all $i\in\mathcal{I}$, and by the continuity of $g_i$ and $h_i$ around $\bar{x}$, we have that 
$$
\varphi_i(x^{k_j}) + \langle \nabla g_i(x^{k_j}) - v_i^{k_j}, d^{k_j} \rangle < 0
$$
for all $i\in\mathcal{I}$ and all $j$ sufficiently large, which clearly contradicts $t^{k_j} > 0$. 
By Lemma~\ref{lem:dk_foc}, there exist $\lambda_{k_j,i} \in [0, 1]$ for $i\in\mathcal{I}$ such that the inclusion 
\[
0 \in \frac{1}{p_{k_j}}\left(\nabla g_0(x^{k_j})-v_0^{k_j}\right) + \sum_{i\in\mathcal{I}} \lambda_{k_j,i} \left(\nabla g_i(x^{k_j})-v_i^{k_j} \right) + G_{k_j} d^{k_j}+\mathcal N_X \left(x^{k_j}+d^{k_j}\right)
\]
is satisfied, and the complementarity condition \eqref{eq:multiplier} holds for $k = k_j$ and all $i\in\mathcal{I}$. 

Since $t^{k_j} > 0$, there exists an index $i_{k_j}\in\mathcal{I}$ such that $\varphi_{i_{k_j}}(x^{k_j}) + \langle \nabla g_{i_{k_j}}(x^{k_j}) - v_{i_{k_j}}^{k_j}, d^{k_j} \rangle > 0$. The complementarity condition \eqref{eq:multiplier} then yields $\lambda_{k_j,i_{k_j}} = 1$, which implies that $\sum_{i\in\mathcal{I}} \lambda_{k_j,i} \ge 1$. Furthermore, the complementarity condition \eqref{eq:multiplier} implies that for all $i\in\mathcal{I}$
\[
 \lambda_{k_j,i} \left(\varphi_{i}(x^{k_j}) + \langle \nabla g_{i}(x^{k_j}) - v_{i}^{k_j}, d^{k_j} \rangle \right) \ge 0.
\]
Next, note that $d^{k_j} \rightarrow 0$ by Proposition~\ref{prop:tau-lower-bound}, and $G_k$ is bounded by Assumption \ref{ass_G}. Since $\{\lambda_{k_j,i}\}$ are bounded within $[0, 1]$, we can pass to a further subsequence to assume that $\lambda_{k_j,i} \to \bar{\lambda}_i \in [0, 1]$ for each $i\in\mathcal{I}$. Because $\sum_{i\in\mathcal{I}} \lambda_{k_j,i} \ge 1$ for all $j$, it follows that $\sum_{i\in\mathcal{I}} \bar{\lambda}_i \ge 1$, ensuring that not all $\bar{\lambda}_i$ are equal to zero.

Taking \( j \to \infty \) in the above inclusion and the inequality,  we obtain the existence of $\bar{\lambda}_i \in [0,1]$ for $i\in\mathcal{I}$, not all zero, such that
\[
0 \in \sum_{i\in\mathcal{I}} \bar{\lambda}_i \left( \nabla g_i(\bar{x})- \partial h_i(\bar{x}) \right) + \mathcal N_X\left(\bar{x}\right), \quad \text{and} \quad \bar{\lambda}_i \varphi_{i}(\bar{x}) \ge 0, \quad \forall i\in\mathcal{I},
\]
which contradicts the imposed ENNAMCQ and thus completes the proof of the proposition.
\end{proof}\vspace*{0.05in}

\section{Sequential Convergence of CDP-ESQM under PLK Property}\label{sec:plk}
\setcounter{equation}{0}
In this section, we establish efficient conditions for the {\em global sequential convergence } of CDP-ESQM that is based on the {\em Polyak-\L ojasiewicz-Kurdyka} properties. This terminology has been recently coined in \cite{bento2025convergence}, where the reader can find detailed historical remarks and discussions. Other names for such conditions and their special cases formulated in terms of various subgradient mappings and their smooth counterparts are used in \cite{Attouch2013,Bolte2007,karimi} and the references therein.

\begin{definition}
Given an l.s.c.\ function $f:\mathbb{R}^n\to (-\infty,\infty]$,
we say that $f$ satisfies the $($basic$)$ {\em Polyak-\L ojasiewicz-Kurdyka} $($PLK$)$ property at $\bar{x} \in \mathrm{dom}(f)$ if there exist positive numbers  $\varepsilon, \eta$ and a concave function $\psi : [0,\eta) \rightarrow [0, \infty)$ such that

(a) $\psi(0) = 0$,

(b) $\psi$ is ${\cal C}^1$ on $(0,\eta)$ and continuous at $0$,

(c) $\psi' > 0$ on $(0,\eta)$, \\[1ex] 
and the following subdifferential inequality 
\[
\psi'(f(x) - f(\bar{x})) \mathrm{dist}(0, {\partial} f(x)) \ge 1
\]
holds whenever $x \in\{ x \in \mathbb{R}^n \mid  \|x - \bar{x}\| < \varepsilon, ~~ f(\bar{x}) < f(x) <f(\bar{x}) +\eta \}$.
A function $f$ satisfying the PKL condition at any point in $\mathrm{dom}(f)$ is called a {\em PKL function}.
\end{definition}

To establish the global sequential convergence of CDP-ESQM, we assume in what follows that the sequence of iterates \( \{x^k\} \) generated by CDP-ESQM is bounded. To proceed further, we construct an appropriate merit function. Since $ \{x^k\} $ is bounded, there exists a compact convex set $D \subset X$ such that $\{x^k\} \subset D$. According to our standing assumptions, {the functions $h_i$ ($i \in \mathcal I$)} are locally Lipschitz continuous and prox-regular over $X$, which implies that {they} are lower-$\mathcal{C}^2$, i.e., they can be expressed as the difference of a convex function and a quadratic function around any point in $X$. The compactness of $D$ allows us to select its finite covering and then, by using \cite[{Theorem~7.102}]{mordukhovich2022convex}, to find $\rho_h > 0$ such that the functions $h_i(x) + \frac{\rho_h}{2}\|x\|^2$, for {$i \in \mathcal I$}, are convex on $D$. Define the approximations
\[
\hat g_i(x): = g_i(x) + \frac{\rho_h}{2}\|x\|^2, \quad \hat{h}_i(x) = h_i(x) + \frac{\rho_h}{2}\|x\|^2, \quad {i \in \mathcal I,}
\]
and then construct the parametric family of {\em merit functions} by
\begin{equation}\label{eq: E definition}  
{
\begin{aligned}
E_{p}(x,\tau,d,G,\mathbf{w},\mathbf{\hat{v}}):=&\frac{1}{p}\left(\hat g_0(x)+\langle w_0,\tau d\rangle+\hat{h}_0^*(\hat v_0)-\langle \hat{v}_0,x\rangle\right)\\
&+\sum_{i\in\mathcal I} \max\left\{\hat g_i(x)+\langle w_i, \tau d\rangle+\hat{h}_i^*(\hat v_i)-\langle \hat{v}_i,x\rangle,0\right\}+\tau\langle Gd,d \rangle+\delta_X(x+d),
\end{aligned}
}
\end{equation}
where {$\mathbf{w} = (w_i)_{i\in\{0\}\cup\mathcal I}$, $\mathbf{\hat{v}} = (\hat{v}_i)_{i\in\{0\}\cup\mathcal I}$}, and $\hat{h}_i^*$ denotes the {\em Fenchel conjugate} of $\hat{h}_i(x) + \delta_D(x)$ given by 
$$
\hat{h}_i^*(\hat v):=\sup_{x \in D} \left\{\langle \hat v,x\rangle-\hat{h}_i(x)\right\}, \quad i \in \mathcal I.
$$

{
Notice that to enable the sequential convergence of the generalized CDP-ESQM algorithm under the PLK property with an iteration-varying matrix $G_k$, we designed the merit function $E_{p}$ \eqref{eq: E definition}
to explicitly incorporate $G$ as a parameter. This technique is inspired by \cite{panshaohua2024proximalnewton}. Moreover, to establish sequential convergence when employing non-unit step sizes generated by a line search procedure, the merit function is also designed to explicitly incorporate the step size $\tau$ as a parameter.
}

Having the sequences $\{x^k\}$, $\{d^k\}$, {and $\{v_i^k\}$ ($i\in \{0\} \cup \mathcal I$)} generated by CDP-ESQM, denote $w_i^k := \nabla g_i(x^k) - v_i^k$ and $\hat{v}_i^k:=v_i^k+\rho_h x^k$ for {$i \in \{0\} \cup \mathcal I$}. {Letting $\mathbf{w}^k = (w_i^k)_{i\in\{0\}\cup\mathcal I}$ and $\mathbf{\hat{v}}^k = (\hat{v}_i^k)_{i\in\{0\}\cup\mathcal I}$}, the next lemma verifies the sufficient decrease condition and the relative error condition for each merit function $E_{p}(x,\tau,d,G,{\mathbf{w},\mathbf{\hat{v}}})$ evaluated along the algorithmic sequence $\{(x^k, \tau_k,d^k, G_k,{\mathbf{w}^k, \mathbf{\hat{v}}^k})\}$.

\begin{lemma}\label{lemma: error condition of E}
Suppose that CDP-ESQM does not terminate in a finite number of steps, and let $\{x^k\}$, $\{d^k\}$, and {$\{v_i^k\}$ ($i\in\{0\}\cup\mathcal I$)} be the sequences generated by CDP-ESQM. 
Assume that both sequences $\{x^k\}$ and $\{p_k\}$ are bounded. Then, for all sufficiently large $k$, we have $p_k = \bar{p}$ with some positive constant $\bar{p}$. Defining further {$w_i^k := \nabla g_i(x^k) - v_i^k$ for $i\in\mathcal I$ and the vectors $\mathbf{w}^k = (w_i^k)_{i\in\{0\}\cup\mathcal I}$, $\mathbf{\hat{v}}^k = (\hat{v}_i^k)_{i\in\{0\}\cup\mathcal I}$}, and assuming that the sequence $\{\mathbf{w}^k\}$ is bounded, there exist positive constants $a$ and $b$ such that the following conditions hold for all sufficiently large $k$:
\begin{gather}
E_{\bar{p}}(x^{k+1}, \tau_{k+1}, d^{k+1}, G_{k+1}, {\mathbf{w}^{k+1}, \mathbf{\hat{v}}^{k+1}}) + {a \|x^{k+1} - x^k\|^2} \le E_{\bar{p}}(x^k, \tau_k, d^k, G_k, {\mathbf{w}^k, \mathbf{\hat{v}}^k}),  \label{eq: suf_decrease} \\
\mathrm{dist}\left(0, \partial E_{\bar{p}}(x^{k}, \tau_{k}, d^{k}, G_{k}, {\mathbf{w}^{k}, \mathbf{\hat{v}}^{k}}) \right) \le {b\|x^{k+1} - x^k\|}. \label{eq: rel_err}
\end{gather}
\end{lemma}

\begin{proof}\,
By the penalty update rule, the boundedness of $\{p_k\}$ yields $p_k = \bar{p}$ with some constant $\bar{p}$ when $k$ is large enough. 
Remembering that $\hat{h}_i^*$ is the Fenchel conjugate of $\hat{h}_i + \delta_D$ and that $\hat{v}_i^k \in \partial \hat{h}_i(x^k) + \mathcal{N}_D(x^k)$ gives us
\begin{equation*}
-\hat{h}_i(x^k)=\hat{h}_i^*(\hat{v}_i^k)-\langle \hat{v}_i^k,x^k\rangle, \quad {i\in\{0\}\cup\mathcal I,}
\end{equation*}
which implies in turn the following relationships valid for all large $k$:
\begin{equation}\label{eq:E_k}
\begin{aligned}
& E_{\bar p}(x^k,\tau_k,d^k,G_k,{\mathbf{w}^k,\mathbf{\hat{v}}^k})\\
  =\, &\frac{1}{\bar p}\left(\hat{g}_0(x^k)+\langle w_0^k,\tau_kd^k\rangle-\hat h_0(x^k)\right)+{\sum_{i\in\mathcal I}}\max\left\{\hat g_i(x^k)+\langle w_i^k,\tau_kd^k\rangle-\hat h_i(x^k),0\right\}+ \tau_k \langle G_kd^k,d^k\rangle\\
  = \, &\tilde\phi_k(x^k+\tau_kd^k)+ \tau_k \langle G_kd^k,d^k\rangle\\
  = \, &\tilde\phi_k(x^k+\tau_kd^k)+{(1-\sigma)\alpha\tau_k\|d^k\|^2}+\tau_k\langle  G_kd^k,d^k\rangle-{(1-\sigma)\alpha\tau_k\|d^k\|^2}\\
  \geq\, &\phi_{\bar p}(x^k+\tau_kd^k)+\tau_k\left(\langle  G_kd^k,d^k\rangle-{(1-\sigma)\alpha}\|d^k\|^2\right) {\ge \phi_{\bar p}(x^{k+1}) + \sigma\alpha\tau_k\|d^k\|^2,}
\end{aligned} 
\end{equation}
where the first inequality follows from the line search condition \eqref{eq:linesearch}, and the second uses $\langle G_k d^k, d^k \rangle \ge \alpha\|d^k\|^2$ from Assumption \ref{ass_G}. Similarly, evaluating $E_{\bar p}$ at the $(k+1)$-th iteration yields
\begin{equation}\label{eq:E_k+1}
\begin{aligned}
&E_{\bar p}(x^{k+1},\tau_{k+1},d^{k+1},G_{k+1},{\mathbf{w}^{k+1},\mathbf{\hat{v}}^{k+1}})\\
  =\, &\tilde\phi_{k+1}(x^{k+1}+\tau_{k+1}d^{k+1}) +\tau_{k+1}\langle G_{k+1}d^{k+1},d^{k+1} \rangle\\
  \leq \,& \phi_{\bar p_{k+1}}(x^{k+1}) - \tau_{k+1}\langle G_{k+1}d^{k+1},d^{k+1} \rangle  +\tau_{k+1}\langle G_{k+1}d^{k+1},d^{k+1} \rangle  = \phi_{\bar p}(x^{k+1}),
\end{aligned} 
\end{equation}
where the inequality follows from \eqref{lem36_eq2} (evaluated at $k+1$) as established in Lemma \ref{lemma:sufficient_decrease}. Combining \eqref{eq:E_k} and \eqref{eq:E_k+1} gives
\begin{equation*}
    \begin{aligned}
        &E_{\bar p}(x^k,\tau_k,d^k,G_k,{\mathbf{w}^k,\mathbf{\hat{v}}^k})-E_{\bar p}(x^{k+1},\tau_{k+1},d^{k+1},G_{k+1},{\mathbf{w}^{k+1},\mathbf{\hat{v}}^{k+1}})\\
        &\geq {\sigma\alpha\tau_k\|d^k\|^2 = \sigma\alpha\tau_k \left\| \frac{x^{k+1}-x^k}{\tau_k} \right\|^2 = \frac{\sigma\alpha}{\tau_k}\|x^{k+1}-x^k\|^2}\\
        &\geq{\sigma\alpha\|x^{k+1}-x^k\|^2},
    \end{aligned}
\end{equation*}
{since $\tau_k \le 1$. By letting $a = \sigma\alpha$, this justifies the claimed sufficient decrease condition \eqref{eq: suf_decrease}.}

To proceed further with verifying the relative error condition \eqref{eq: rel_err}, we employ the limiting subdifferential calculus rules for finite sums and finite maxima taken from \cite[Theorem~2.19 and 4.10]{mordukhovich2018variational}, together {with} the fact that the function $\hat g_i(x)+\langle w_i,\tau d\rangle+\hat{h}_i^*(\hat v_i)-\langle \hat{v}_i,x\rangle$ under the pointwise maximum operator is lower regular. The latter is due to the smoothness of $\hat g_i$ and the convexity of $\hat{h}_i^*$.
Based on these considerations, the limiting subdifferential of $E_{p}(x,\tau,d,G,{\mathbf{w},\mathbf{\hat{v}}})$ is calculated by
\begin{equation*}
    \begin{aligned}
        \partial E_p&(x,\tau,d,G,{\mathbf{w},\mathbf{\hat{v}}}) =\\
        &\left\{
\begin{bmatrix}
\xi_x\\ \xi_\tau\\ \xi_d \\ \xi_G \\ {\xi_{\mathbf{w}}}\\ {\xi_{\mathbf{\hat{v}}}}
\end{bmatrix} 
~\left| ~\begin{array}{l}
\xi_x \in \frac{1}{p}\left(\nabla g_0(x) - {\hat{v}_0} \right) + {\sum_{i\in\mathcal I} \lambda_i}\left( \nabla g_i(x) - {\hat{v}_i}\right) + \mathcal{N}_X(x+d),\\
\xi_\tau = \frac{1}{p}\langle w_0, d \rangle + {\sum_{i\in\mathcal I} \lambda_i} \langle w_i, d \rangle + \langle Gd,d\rangle, \\
\xi_d \in \frac{\tau}{p}w_0 + {\sum_{i\in\mathcal I} \lambda_i} \tau w_i + 2\tau G d + \mathcal{N}_X(x+d), \\
\xi_G = {\tau} dd^T, \quad {\xi_{w_0} = \frac{\tau}{p}d, \quad \xi_{w_i} = \lambda_i \tau d \ (i\in\mathcal I)}, \\
{\xi_{\hat{v}_0}} \in \frac{1}{p}\left(\partial \hat{h}_0^*({\hat{v}_0})-x\right), \quad {\xi_{\hat{v}_i}} \in {\lambda_i}\left(\partial \hat{h}_i^*({\hat{v}_i})-x\right) {\ (i\in\mathcal I)}, \\
{\lambda_i} \in [0,1], {\lambda_i} \min\left\{\hat{g}_i(x)+\langle w_i,\tau d\rangle+\hat{h}_i^*({\hat{v}_i})-\langle {\hat{v}_i},x\rangle, 0 \right\} = 0, \\ 
\quad\,\,\,\, (1- {\lambda_i})\max\left\{ \hat{g}_i(x)+\langle w_i,\tau d\rangle+\hat{h}_i^*({\hat{v}_i})-\langle {\hat{v}_i},x\rangle, 0 \right\} = 0 {\ (i\in\mathcal I)}
\end{array}\right. \right\}.
    \end{aligned}
\end{equation*}
Since $d^k$ solves the convex subproblem \eqref{direction_sub}, there {exist $\lambda_{k,i} \in \left[0, 1\right]$} satisfying \eqref{eq:foc} together with the equalities in \eqref{eq:multiplier}. Furthermore, the definition of $\hat{v}_i$ and the result of \cite[Theorem~4.20]{beck2017first} tell us that $x^k\in\partial\hat{h}_i^*(\hat v_i^k)$ for {$i \in \{0\}\cup\mathcal I$}.
Therefore, we arrive at the inclusion

\begin{equation*}
  \begin{aligned}
     &\left( 
     -G_kd^k, 
     \big\langle \frac{1}{\bar p}w_0^k + \sum_{i \in \mathcal{I}} \lambda_{k,i} w_i^k, d^k \big\rangle + \langle G_kd^k, d^k \rangle, 
     (2\tau_k - 1)G_k d^k, 
     \tau_k d^k (d^k)^T, 
     \tau_k \frac{d^k}{\bar p}, 
     ( \tau_k \lambda_{k,i} d^k )_{i \in \mathcal{I}}, 
     \mathbf{0}
     \right) \\ 
     &\subset \partial E_{\bar p}(x^k, \tau_k, d^k, G_k, \mathbf{w}^k, \mathbf{\hat{v}}^k)
  \end{aligned}  
\end{equation*}
Combining this with the facts that $\tau_k \in (0,1]$ and $\lambda_{k,i} \in [0,1]$ for $i \in \mathcal{I}$, yields the following estimate for all sufficiently large $k$ (where $p_k = \bar{p}$)
\begin{equation*}
    \begin{aligned}
        &\mathrm{dist}\left(0, \partial E_{\bar p}(x^k,\tau_k,d^k,G_k,\mathbf{w}^k,\mathbf{\hat{v}}^k) \right) \\
        \le \, & \left( 2\|G_k\|+ \frac{1}{\bar p}\| w^k_0\|+\sum_{i=1}^m\| w^k_i\| + \|G_k\|\|d_k\| + \|d_k\| + 1/\bar{p} + \sum_{i \in \mathcal{I}}\lambda_{k,i} \right) \|d^k\|.
    \end{aligned}
\end{equation*}
Notice that $ \inf_k \tau_k =: \underline{\tau} > 0$ by Proposition~\ref{prop:tau-lower-bound}. Furthermore, $\| x^{k+1} - x^k \| \rightarrow 0$ by Proposition \ref{prop:bounded-subsequence}, which implies $\|d_k\| = \| x^{k+1} - x^k \|/\tau_k \le \| x^{k+1} - x^k \|/\underline{\tau}  \rightarrow 0$. Since the sequence $\{\mathbf{w}^k\}$ is bounded by assumption, and $G_k$ is bounded by Assumption \ref{ass_G}, we obtain the existence of $b > 0$ such that the relative error condition \eqref{eq: rel_err} holds. This completes the proof.
\end{proof}\vspace*{0.05in}

Having in hand the sufficient decrease condition \eqref{eq: suf_decrease} and the relative error condition \eqref{eq: rel_err} for $	E_{p}$ obtained above, we now establish the global convergence of CDP-ESQM under the basic PLK property. The proof mainly follows the scheme of applying PLK conditions outlined in \cite{Attouch2013,bento2025convergence}.

\begin{theorem}\label{thm: KL_convergence}
Let  \( \{x^k\} \)  and \( \{p_k\} \) be the sequences generated by CDP-ESQM. 
If both sequences \( \{x^k\} \) and \( \{p_k\} \) are bounded, then we have $p_k = \bar{p}$  for some positive constant $\bar{p}$ when $k$ is sufficiently large. 
Assuming furtheremore that $E_{\bar{p}}$ from \eqref{eq: E definition} is a PLK function ensures that \( \{x^k\} \) converges to a stationary point of the constrained difference programming problem {\rm(\ref{difference-programming})}.
\end{theorem}

\begin{proof}\,
Since $\{x^k\}$ is bounded, Proposition~\ref{prop:tau-lower-bound} ensures that $\inf_{k}\tau_k >0$, and Proposition \ref{prop:bounded-subsequence} ensures that $\|d^k \| \rightarrow 0$. Define
$$
w_i^k := \nabla g_i(x^k) - v_i^k\;\mbox{ and }\;\hat v_i^k=v_i^k + \rho_h x^k,\quad {i \in \{0\} \cup \mathcal I,}
$$
and denote by $\Omega$ the collection of all limiting points of the sequence {$\{(x^k, \tau_k, d^k, G_k, \mathbf{w}^k, \mathbf{\hat{v}}^k)\}$}. Using the boundedness of $\{x^k\}$, the conditions $w_i^k = \nabla g_i(x^k) - v_i^k$ and $\hat v_i^k=v_i^k + \rho_h x^k$ with $v_i^k \in \partial h_i(x^k)$ for {$i \in \{0\}\cup\mathcal I$}, the continuous differentiability of {$g_i$} on $X$, and the local Lipschitz continuity of {$h_i$} on $X$, it follows that the {sequences $\{\mathbf{w}^{k}\}$ and $\{\mathbf{\hat{v}}^{k}\}$} are also bounded. {Combining this with $d^k \to 0$, $\tau_k \in (0,1]$, and the boundedness of $\{G_k\}$ from Assumption \ref{ass_G}, we conclude that $\{(x^k, \tau_k, d^k, G_k, \mathbf{w}^k, \mathbf{\hat{v}}^k)\}$} is a bounded sequence. Thus $\Omega$ is a bounded set, which is clearly closed, and so we have $\lim_{k \rightarrow \infty} \mathrm{dist}({(x^k, \tau_k, d^k, G_k, \mathbf{w}^k, \mathbf{\hat{v}}^k)}, \Omega) = 0$.

The boundedness of {$\{(x^k, \tau_k, d^k, G_k, \mathbf{w}^k, \mathbf{\hat{v}}^k)\}$} ensures that the sequence {$\{E_{\bar p}(x^k, \tau_k, d^k, G_k, \mathbf{w}^k, \allowbreak \mathbf{\hat{v}}^k)\}$} is also bounded. It follows from the sufficient decrease condition \eqref{eq: suf_decrease} that the latter sequence is nonincreasing. Therefore, there exists a constant $\bar{E}$ such that any subsequence of {$\{E_{\bar{p}}(x^k, \tau_k, d^k, G_k, \mathbf{w}^k, \mathbf{\hat{v}}^k)\}$} converges to $\bar{E}$. This tells us that $E_{\bar{p}}$ is constant on $\Omega$.
Without loss of generality, assume that 
$$
{E_{\bar{p}}(x^k, \tau_k, d^k, G_k, \mathbf{w}^k, \mathbf{\hat{v}}^k)} > \bar{E}\;\mbox{ for all }\; k\in\mathbb N.
$$
If there exists some $k\in\mathbb N$ such that ${E_{\bar{p}}(x^k, \tau_k, d^k, G_k, \mathbf{w}^k, \mathbf{\hat{v}}^k)} = \bar{E}$, then \eqref{eq: suf_decrease} implies that {$\|x^{k+1}-x^k\| = 0$} for all sufficiently large $k$. In this case, the sequence $\{x^k\}$ of CDP-ESQM iterates converges {in finite steps}, and the conclusion of the theorem clearly holds.

It follows from $\lim_{k \rightarrow \infty} \mathrm{dist}({(x^k, \tau_k, d^k, G_k, \mathbf{w}^k, \mathbf{\hat{v}}^k)}, \Omega) = 0$ and {$\lim_{k \rightarrow \infty} E_{\bar{p}}(x^k, \tau_k, d^k, G_k, \mathbf{w}^k, \allowbreak \mathbf{\hat{v}}^k) = \bar{E}$} that for any $\varepsilon, \eta > 0$, there exists $k_0\in\mathbb{N}$ such that 
\begin{equation*}
    \mathrm{dist}({(x^k, \tau_k, d^k, G_k, \mathbf{w}^k, \mathbf{\hat{v}}^k)}, \Omega) < \varepsilon  \text{ and }  \bar{E} < {E_{\bar{p}}(x^k, \tau_k, d^k, G_k, \mathbf{w}^k, \mathbf{\hat{v}}^k)} < \bar{E} + \eta \quad \mbox{whenever } k \ge k_0.
\end{equation*}
Since $E_{\bar{p}}$ is a PLK function, $E_{\bar{p}}$ is constant on $\Omega$, and the set $\Omega$ is compact, it is easy to get the uniform version of the PLK property of $E_{\bar{p}}$ on $\Omega$ saying that there exists a continuous concave function $\psi$ such that for all $k \ge k_0$ we get the inequality
\[
\psi' \left( {E_{ \bar{p}}(x^k, \tau_k, d^k, G_k, \mathbf{w}^k, \mathbf{\hat{v}}^k)} - \bar{E} \right) \mathrm{dist}\big( 0, {\partial} {E_{\bar{p}}(x^k, \tau_k, d^k, G_k, \mathbf{w}^k, \mathbf{\hat{v}}^k)} \big)\ge 1.
\]
Combining the latter with \eqref{eq: suf_decrease} and \eqref{eq: rel_err}, and applying the standard arguments as in \cite{Attouch2013} (which we omit for brevity) brings us to the series convergence
$$
\sum_{k=1}^\infty\| x^{k+1} - x^k \| < \infty,
$$
and thus the sequence $\{x^k\}$ converges. Then the conclusion of the theorem follows from Theorem~\ref{general problem convergence}.
\end{proof}\vspace*{0.05in}

It has been well recognized that the PLK property holds for broad classes of functions that appear in optimization and variational analysis.  A remarkable class of such functions consists of semialgebraic ones.

\begin{definition} A {\em set} $D \subset \mathbb{R}^n$ is called  {\em semialgebraic} if there exists a finite number
of real polynomial functions $g_{ij}, h_{ij}: \mathbb{R}^n \rightarrow \mathbb{R}$ such that
\[
D = \bigcup_{j = 1}^p\bigcap_{i = 1}^q\{x \in \mathbb{R}^n ~|~ g_{ij}(x) = 0, ~ h_{ij}(x) < 0 \}.
\]
A {\em function} $f : \mathbb{R}^n \rightarrow (-\infty, \infty]$ is {\em semialgebraic} if its graph
is a semialgebraic set.
\end{definition}

Recall that the class of semialgebraic functions is invariant with respect to taking sums and compositions, that the indicator function of semialgebraic sets is semialgebraic, and that the partial minimization of
semialgebraic functions over semialgebraic sets is semialgebraic; see, e.g., \cite{Attouch2013}. This allows us to conclude that the merit function \( E_{\bar{p}} \), built in \eqref{eq: E definition} upon the semialgebraic data of the difference programming problem \eqref{difference-programming}, is semialgebraic. Therefore, we get the following efficient consequence of Theorem~\ref{thm: KL_convergence}.

\begin{corollary}\label{cor:semi} Assume in problem \eqref{difference-programming}, the functions $\varphi_0$ and $\varphi_i$ for $i\in\mathcal{I}$, and the set $X$ are semialgebraic, and let $\{x^k\}$ and $\{p_k\}$  be the sequences generated by CDP-ESQM. 
If both sequences $\{x^k\}$ and $\{p_k\}$ are bounded, then $\{x^k\}$  converges to a stationary point of \eqref{difference-programming}.
\end{corollary}

\section{VI-Constrained Difference Programming}\label{sec:VI-constrained}
\setcounter{equation}{0}

In this section, we consider the {\em VI-constrained difference programming problem} defined by
\begin{equation}\label{VI-constrained}
\begin{aligned}
\min_{ y \in Y, z \in \mathbb{R}^{m_2}}  
~~& \Psi(y, z) := \Psi_1(y,z) - \Psi_2(y,z) \\
\text{subject to } ~~~~& z \in S(y) := \left\{
z \in\Gamma(y) \mid \langle F(y, z), \theta - z \rangle \geq 0 \text{ for all } \theta \in \Gamma(y)
\right\},
\end{aligned}
\end{equation}  
where $Y \subset \mathbb{R}^{m_1}$ is nonempty, convex, and closed sets. Throughout the section, the function \( \Psi_1: \mathbb{R}^{m_1} \times \mathbb{R}^{m_2} \to \mathbb{R} \) and mapping $ F: \mathbb{R}^{m_1} \times \mathbb{R}^{m_2} \to \mathbb{R}^{m_2} $ are assumed to be continuously differentiable with locally Lipschitzian  gradients and Jacobians, respectively, on \(  \left( Y \times \mathbb{R}^{m_2} \right) \cap \mathrm{gph} (\Gamma) \), while the function \( \Psi_2: \mathbb{R}^{m_1}\times \mathbb{R}^{m_2} \to \mathbb{R} \) is a locally Lipschitz continuous and prox-regular function over \(  \left( Y \times \mathbb{R}^{m_2} \right) \cap \mathrm{gph} (\Gamma) \). The set-valued mapping $\Gamma(y) : \mathbb{R}^{m_1} \rightrightarrows \mathbb{R}^{m_2} $ is defined by
\begin{equation}\label{Gamma}
\Gamma(y) := \{z \in \mathbb{R}^{m_2} \mid \zeta_i(y,z) \le 0, i =1, \ldots,s\},
\end{equation}
where each $\zeta_i(y,z): \mathbb{R}^{m_1} \times \mathbb{R}^{m_2} \to \mathbb{R}$ is a continuously differentiable and convex function on  \( Y \times \mathbb{R}^{m_2} \).\vspace*{0.05in}

To ensure the well-posedness of problem \eqref{VI-constrained}, we impose the additional standing assumptions.

\begin{assumption}\label{assumption1} 
{\rm The following conditions hold:
\begin{enumerate}
\item[\bf(a)] The objective function $\Psi$ is bounded from below on $ \left( Y \times \mathbb{R}^{m_2} \right) \cap \mathrm{gph} (\Gamma)$.

\item[\bf(b)] There exists an open set $U\subset\mathbb{R}^{m_1}$ such that $Y \subset U$ and $\Gamma(y) \neq \varnothing$ for all $ y \in U$.
\end{enumerate}}
\end{assumption}

A classical approach to solving VI involves reformulating VI as a minimization problem using a gap function. Utilizing the {\em regularized gap function} introduced in \cite{fukushima1992equivalent}, we define
\begin{equation}\label{def_mu}
\mu_\gamma(y,z)	:= \inf\limits_{\theta\in \Gamma(y)}\left\{
\langle F(y,z),\theta\rangle+\frac{1}{2\gamma}\|\theta-z\|^2
\right\},
\end{equation}
where $\gamma >0$. Then $\langle F(y,z),z\rangle - \mu_\gamma(y,z)$ corresponds precisely to the regularized gap function for VI in \eqref{VI-constrained} given \cite{fukushima1992equivalent}. As proved in 
\cite[Theorem~3.1]{fukushima1992equivalent}, the following properties hold:
$$\langle F(y,z),z\rangle - \mu_\gamma(y,z) \ge 0\;\mbox{ for all }\;z \in \Gamma(y),$$ 
$$\langle F(y,z),z\rangle - \mu_\gamma(y,z) \le 0, ~z \in \Gamma(y) \quad \text{if and only if} \quad z \in S(y).$$ Leveraging this equivalence between the solution set of VI and the level set of the regularized gap function yields the single-level reformulation approach introduced in \cite{marcotte1996exact}. Consequently, the VI-constrained difference programming problem can be equivalently reformulated as follows
\begin{equation}\label{VI-constraint_equivalent}
\begin{aligned}
\min_{ y \in Y, z\in \Gamma(y)}  ~~& 
\Psi(y,z)\\
\text{subject to } ~~~~& \langle F(y,z),z\rangle - \mu_\gamma(y,z) \le 0.
\end{aligned}
\end{equation}

Our subsequent analysis demonstrates that reformulation \eqref{VI-constraint_equivalent} constitutes a special case of the constrained difference programming problem \eqref{difference-programming}. Consequently, the proposed 
Algorithm~\ref{alg:ESQM} (CDP-ESQM) can be directly applied to solving \eqref{VI-constraint_equivalent}. To proceed in this way, let us first investigate the properties of the function $\mu_\gamma(y,z)$ defined in \eqref{def_mu}. Specifically, we show that $\mu_\gamma(y,z)$ is prox-regular and then provide an explicit expression for its subdifferential.

Observe that the optimization problem defining $\mu_\gamma(y,z)$ in \eqref{def_mu} is strongly convex.  Its unique solution admits the closed-form expression
\begin{equation}\label{def_theta}
\theta_\gamma^*(y, z) :=  \mathop{\arg\min}\limits_{\theta \in \Gamma(y)} \left\{
\langle F(y,z), \theta \rangle + \frac{1}{2\gamma} \|\theta - z\|^2
\right\} = \mathrm{Proj}_{\Gamma(y)}\left(z - \gamma F(y,z)\right).
\end{equation}
Thus, when  the set $\Gamma(y)$ is simple and the projection onto it is computationally tractable, $\theta_\gamma^*(y, z)$ can be efficiently computed. Consequently, $\mu_\gamma(y,z)$ is also readily computable since it has the expression $$\mu_\gamma(y,z) = \langle F(y,z), \theta_\gamma^*(y, z)  \rangle + \frac{1}{2\gamma} \|\theta_\gamma^*(y, z)  - z\|^2.$$ Furthermore, the following lemma verifies the local boundedness of $\theta_\gamma^*(y, z)$ on $Y \times \mathbb{R}^{m_2}$.

\begin{lemma}\label{lemma:local_boundedness}
The mapping \( \theta_\gamma^*(y, z)\) is locally bounded around every point in $Y \times \mathbb{R}^{m_2}$.
\end{lemma}

\begin{proof}\,
Let $(\bar y, \bar z) \in Y \times \mathbb{R}^{m_2}$. Assume on the contrary that there exists a sequence $\{(y_k,z_k)\}$ such that $(y_k, z_k) \to (\bar y, \bar z)$ as $k \to \infty$ and $\|\theta_\gamma^*(y_k,z_k)\| > k$ for every $k \in \mathbb{N}$. It follows from the continuity of $F$ that
\begin{equation}\label{limit goes to infinity}
\begin{aligned}
\left\langle F(y_k,z_k), \theta^*_\gamma(y_k,z_k) \right\rangle
+ \frac{1}{2\gamma} \left\|\theta_\gamma^*(y_k,z_k) - z_k\right\|^2
\to \infty\;\text{ as }\;k \to \infty.
\end{aligned}
\end{equation}
Consider the value function $\nu(y) := \inf_\theta \{ \|\theta\| \mid \theta \in \Gamma(y)\}$ and deduce from the convexity of $\zeta_i(y,z)$ for $i = 1,\ldots,s$ on  \( Y \times \mathbb{R}^{m_2} \) that $\nu(y)$ is convex on $Y$; see \cite[Theorem 2.129]{mordukhovich2022convex}. The imposed assumption tells us that  $\Gamma(y) \neq \varnothing$ on an open set $U$ with $Y \subset U$, which yields $\nu(y) < \infty$ on $U$. Consequently, $\mathrm{dom}(\nu) \subset U$, and it follows from \cite[Corollary 2.152]{mordukhovich2022convex} that
$$
| \nu(y_k) -\nu(\bar{y})| \le L_\theta\|y_k - \bar{y}\|
$$ 
for all $k$ any sufficiently large. Therefore, for such $k$ we can find
$\theta_k \in \Gamma(y_k)$ satisfying 
$$
\|\theta_k\| \le \| \theta^*_\gamma(\bar{y}, \bar{z})\| + L_\theta\|y_k - \bar{y}\|.
$$
This leads us to the estimates
\begin{equation*}
\begin{aligned}
&\left\langle F(y_k,z_k), \theta^*_\gamma(y_k,z_k) \right\rangle
+ \frac{1}{2\gamma} \left\|\theta_\gamma^*(y_k,z_k) - z_k\right\|^2 \\
&\leq\left\langle  F(y_k,z_k), \theta_k \right\rangle
+ \frac{1}{2\gamma} \left\|\theta_k - z_k\right\|^2 \\
&\leq\left\| F(y_k,z_k) \right\| \left(\| \theta^*_\gamma(\bar{y}, \bar{z})\| + L_\theta\|y_k - \bar{y}\| \right) + \frac{1}{\gamma} \left( 2\| \theta^*_\gamma(\bar{y}, \bar{z})\|^2 + 2L_\theta^2\|y_k - \bar{y}\|^2 + \left\|z_k\right\|^2 \right).
\end{aligned}
\end{equation*}
However, the right-hand side above remains bounded as  $k \to \infty$, which contradicts \eqref{limit goes to infinity}. 
\end{proof}\vspace*{0.07in}

Next we verify the prox-regularity of $\mu_\gamma(y,z)$ and provide an explicit formula for its subdifferential.
Observe that the normal cone representation \eqref{thm_proxreg_eq5} used below holds, in particular, when the functions $\zeta_i$ in \eqref{Gamma} satisfy the classical Mangasarian-Fromowitz constraint qualification; see, e.g., \cite[Theorem~6.14]{rockafellar1998variational}.

\begin{proposition}\label{prop:prox-regularity} 
For any $(\bar{y}, \bar{z}) \in Y \times \mathbb{R}^{m_2}$, there exist constants $\varepsilon, \rho_{\mu} > 0$ such that the function $$\mu_\gamma(y, z) + \frac{\rho_{\mu}}{2}\|(y, z)\|^2$$ is convex on $\mathbb{B}_{\varepsilon}(\bar{y}, \bar{z})$. Consequently, $\mu_\gamma(y, z)$ is prox-regular at \( (\bar{x}, \bar{y}) \). Furthermore, if the  condition
\begin{equation}\label{thm_proxreg_eq5}
\mathcal{N}_{\mathrm{gph}(\Gamma)}(y,z) = \left\{ \left.\sum_{i = 1}^{s} \nu_i \nabla \zeta_i({y}, z)  ~\right|~ \nu \in \mathbb{R}^s,\;\nu \ge 0,\;\sum_{i = 1}^{s}\nu_i \zeta_i({y},z ) = 0   \right\}
\end{equation}  
holds at $(\bar{y}, \theta_\gamma^*(\bar y, \bar{z}))$, then the subdifferential $\partial\mu_\gamma(\bar{y}, \bar{z})$ of $ \mu_\gamma$ at $(\bar{y}, \bar{z})$ is computed by
\[
\left\{ \left. \begin{bmatrix}
\nabla_y F(\bar{y}, \bar{z})^T \bar\theta^* + \displaystyle	\sum_{i = 1}^{s} \nu_i \nabla_y \zeta_i(\bar{y}, \bar\theta^*) \\
\nabla_z F(\bar{y}, \bar{z})^T \bar\theta^* + 
\displaystyle \frac{\bar{z} - \bar\theta^*}{\gamma}
\end{bmatrix}	~\right|~ 
\begin{array}{c}
\nu \in \mathbb{R}^s, \nu \ge 0, \displaystyle\sum_{i = 1}^{s}\nu_i \zeta_i(\bar{y}, \bar\theta^* ) = 0, \\
F(\bar y,\bar z)+\displaystyle\frac{\bar\theta^*-\bar z}{\gamma}+\displaystyle\sum_{i=1}^s\nu_i\nabla_z\zeta_i(\bar y,\bar\theta^*)=0
\end{array}
\right\},
\]
where $\bar\theta^*=\theta_\gamma^*(\bar x,\bar y)$, $\nabla_y F(y,z)$ and $\nabla_z F(y,z)$ denotes the partial Jacobian of $F$ with respect to $y$ and $z$, respectively, and where $\nabla_y \zeta_i(y,z)$ stands for the partial gradient of $\zeta_i$ with respect to the variable $y$.
\end{proposition}

\begin{proof}\, Since both $F$ and $\nabla F$ are assumed to be locally Lipschitzian around $(\bar{y}, \bar{z})$, there exist constants $L_F, L_{\nabla F}, \varepsilon > 0$ such that $F$ and $\nabla F$ are $L_F$- and $L_{\nabla F}$-Lipschitz continuous on $\mathbb{B}_{\varepsilon}(\bar{y}, \bar{z})$, respectively. It follows from 
Lemma~\ref{lemma:local_boundedness} that \( \theta_\gamma^*(y, z)\) is locally bounded around $(\bar{y}, \bar{z})$, and so we can assume without loss of generality that there exists $M_\theta > 0$ for which $ \| \theta_\gamma^*(y, z) \| \le M_{\theta}$ whenever $(y,z) \in \mathbb{B}_{\varepsilon}(\bar{y}, \bar{z})$.
	
We aim further to demonstrate that the function
\begin{equation}
\mu_\gamma(y, z) + \left(  \gamma L_F^2 + L_{\nabla F}M_\theta + \frac{1}{2\gamma}\right)\|(y, z)\|^2
\end{equation}
is convex on $\mathbb{B}_{\varepsilon}(\bar{y}, \bar{z})$. Picking any $(y,z) \in \mathbb{B}_{\varepsilon}(\bar{y}$ and using $\bar{z})$, since $ \| \theta_\gamma^*(y, z) \| \le M_{\theta}$ give us
\begin{equation}\label{thm_proxreg_eq2}
\begin{aligned}
& \mu_\gamma(y, z) + \left( \gamma L_F^2 + L_{\nabla F}M_\theta + \frac{1}{2\gamma} \right)\|(y, z)\|^2 \\
=&\inf\limits_{\theta} \left\{  
\langle F(y,z), \theta \rangle + \frac{1}{2\gamma} \|\theta - z\|^2 + \left( \gamma L_F^2 + L_{\nabla F}M_\theta + \frac{1}{2\gamma}\right)\|(y,z)\|^2 + \delta_{\mathrm{gph}(\Gamma)}(y, \theta) 	\right\} \\
=&\inf\limits_{\theta} \left\{  \left.
\langle F(y,z), \theta \rangle + \frac{1}{2\gamma} \|\theta - z\|^2 + \left( \gamma L_F^2 + L_{\nabla F}M_\theta + \frac{1}{2\gamma}\right)\|(y,z)\|^2 + \delta_{\mathrm{gph}(\Gamma)}(y, \theta) 
~\right|~ \| \theta \| \le 2M_{\theta} 	\right\}. 
\end{aligned}
\end{equation}
To justify the claimed convexity, let us first show that the function
\begin{equation}\label{Fconvex}
\langle F(y,z), \theta \rangle + (\gamma L_F^2+ L_{\nabla F}M_\theta)
\|(y,z)\|^2
+ \frac{1}{4\gamma}\|\theta\|^2
\end{equation}
is convex with respect to $(y,z, \theta)$ on the set $\mathbb{B}_{\varepsilon}(\bar{y}, \bar{z}) \times \mathbb{B}_{ 2M_{\theta}}$. This is achieved by verifying the monotonicity of its gradient. Specifically, we intend to check that the function
\begin{equation}\label{thm_proxreg_eq1}
\left\langle \begin{bmatrix}
\nabla_y F(y,z)^T\theta - \nabla_y F(y',z')^T \theta' \\  
\nabla_z F(y,z)^T\theta - \nabla_z F(y',z')^T \theta' \\  
F(y,z) - F(y',z')
\end{bmatrix},
\begin{bmatrix}
y - y' \\ z-z' \\\theta - \theta'  
\end{bmatrix} \right\rangle + (2\gamma L_F^2 + 2L_{\nabla F}M_\theta)
\|(y,z) - (y',z')\|^2
+ \frac{1}{2\gamma}\|\theta - \theta'\|^2.
\end{equation}
is nonnegative for any $(y,z,\theta), (y',z',\theta') \in \mathbb{B}_{\varepsilon}(\bar{y}, \bar{z}) \times \mathbb{B}_{ 2M_{\theta}}$. Indeed, picking such triples brings us to 
\begin{equation*}
\begin{aligned}
&\left\langle 
\begin{bmatrix}
\nabla_y F(y,z)^T\theta - \nabla_y F(y',z')^T \theta' \\  
\nabla_z F(y,z)^T\theta - \nabla_z F(y',z')^T \theta' \\  
\end{bmatrix},
\begin{bmatrix}
y - y' \\ z-z' 
\end{bmatrix} \right\rangle \\
&=\left\langle
\begin{bmatrix}
\nabla_y F(y,z)^T\theta - \nabla_y F(y',z')^T \theta \\  
\nabla_z F(y,z)^T\theta - \nabla_z F(y',z')^T \theta \\  
\end{bmatrix},
\begin{bmatrix}
y - y' \\ z-z' 
\end{bmatrix} \right\rangle + \left\langle 
\begin{bmatrix}
\nabla_y F(y',z')^T\theta - \nabla_y F(y',z')^T\theta' \\
\nabla_z F(y',z')^T\theta - \nabla_z F(y',z')^T\theta'
\end{bmatrix},
\begin{bmatrix}
y - y' \\ z-z' 
\end{bmatrix} \right\rangle\\
&\ge -L_{\nabla F}\|\theta\| \cdot \|(y,z) - (y',z')\|^2 - L_F\|\theta - \theta'\| \cdot \|(y,z) - (y',z')\|\\
&\ge -2L_{\nabla F}M_\theta\|(y,z) - (y',z')\|^2 - \frac{1}{4\gamma}\|\theta - \theta'\|^2 - \gamma L_F^2
\|(y,z) - (y',z')\|^2
\end{aligned}
\end{equation*}
together with the estimates
\begin{equation*}
\left\langle  
F(y,z) - F(y',z'), \theta - \theta'
\right\rangle  
\geq 
-L_F\|\theta - \theta'\|\cdot \|(y,z) - (y',z')\| \ge  - \frac{1}{4\gamma}\|\theta - \theta'\|^2 - \gamma L_F^2
\|(y,z) - (y',z')\|^2.
\end{equation*}
Combining the above yields the desired inequality \eqref{thm_proxreg_eq1}. Consequently, it follows from \cite[Theorem~12.17]{rockafellar1998variational} that the function defined in \eqref{Fconvex} is convex with respect to $(y,z, \theta)$ on $\mathbb{B}_{\varepsilon}(\bar{y}, \bar{z}) \times \mathbb{B}_{ 2M_{\theta}}$.

By taking into account the equality
$$
\frac{1}{2\gamma} \|\theta - z\|^2 + \frac{1}{2\gamma} \|(y,z)\|^2 - \frac{1}{4\gamma} \|\theta\|^2
= \frac{1}{2\gamma}\|y\|^2 + \frac{1}{4\gamma}\left\|2z - \theta\right\|^2,
$$
we readily conclude that the function
\[
\frac{1}{2\gamma} \|\theta - z\|^2 + \frac{1}{2\gamma} \|(y,z)\|^2 - \frac{1}{4\gamma} \|\theta\|^2
\]
is convex with respect to $(y,z, \theta)$,  Combining  now these convexity results with the convexity of the graphical set $\mathrm{gph}(\Gamma)(y,z)$ tells us that the function
\[
\langle F(y,z), \theta \rangle + \frac{1}{2\gamma} \|\theta - z\|^2 + \left( \gamma L_F^2 + L_{\nabla F}M_\theta + \frac{1}{2\gamma}\right)\|(y,z)\|^2 + \delta_{\mathrm{gph}(\Gamma)}(y, \theta) 
\]
is convex with respect to $(y,z, \theta)$ on $\mathbb{B}_{\varepsilon}(\bar{y}, \bar{z}) \times \mathbb{B}_{ 2M_{\theta}}$. By using \cite[Proposition~1.50]{mordukhovich2023easy}, we can deduce from the representations in  \eqref{thm_proxreg_eq2} that the function
\begin{equation}\label{mu-convex}
\mu_\gamma(y, z) + \left(  \gamma L_F^2 + L_{\nabla F}M_\theta + \frac{1}{2\gamma}\right)\|(y, z)\|^2
\end{equation}
is convex on $\mathbb{B}_{\varepsilon}(\bar{y}, \bar{z})$. This therefore verifies the prox-regularity of $\mu_\gamma(y, z)$ at $(\bar{y}, \bar{z})$.

It follows from the convexity of the function in \eqref{mu-convex} on $\mathbb{B}_{\varepsilon}(\bar{y}, \bar{z})$ that
\begin{equation}\label{thm_proxreg_eq3}
\begin{array}{ll}
\disp\partial \mu_\gamma(\bar{y}, \bar{z})& = \partial \left(  \mu_\gamma(\bar{y}, \bar{z}) + \left(  \gamma L_F^2 + \disp\frac{1}{2}L_{\nabla F}M_\theta + \frac{1}{2\gamma}\right)\|(\bar{y}, \bar{z})\|^2 \right)\\
&-\disp\left(  \gamma L_F^2 + \frac{1}{2}L_{\nabla F}M_\theta + \frac{1}{2\gamma}\right) \nabla \|(\bar{y}, \bar{z})\|^2.
\end{array}
\end{equation}
Using \eqref{thm_proxreg_eq2} and the subdifferential formula for the partial minimization of convex functions established in \cite[Theorem~2.61]{mordukhovich2023easy} along with the subdifferential calculus rule 
\cite[Theorem~3.18]{mordukhovich2023easy}, we get
\begin{equation}\label{thm_proxreg_eq4}
\begin{aligned}
&\partial \left(  \mu_\gamma(\bar{y}, \bar{z})+ \left(  \gamma L_F^2 + \frac{1}{2}L_{\nabla F}M_\theta + \frac{1}{2\gamma}\right)\|(\bar{y}, \bar{z})\|^2 \right) \\
= &\left\{ (\xi_1, \xi_2) ~\left| ~
\begin{bmatrix}
\xi_1 \\ \xi_2 \\0 
\end{bmatrix} \!\in\! 
\begin{bmatrix}
\nabla_y F(\bar{y}, \bar{z})^T \theta \\ \nabla_z F(\bar{y}, \bar{z})^T \theta \\ F(\bar{y}, \bar{z})
\end{bmatrix} \!+\! 
\begin{bmatrix}
\left(2\gamma L_F^2 + 2L_{\nabla F}M_\theta + 1/\gamma \right) \bar{y} \\
\left(2\gamma L_F^2 + 2L_{\nabla F}M_\theta + 1/\gamma \right) \bar{z} + (\bar{z} - \theta)/\gamma \\
(\theta - \bar{z})/\gamma
\end{bmatrix} \!+\! \mathcal{N}_{\mathcal{F}}(\bar y, \bar{z}, \theta)
\right.\right\},
\end{aligned}
\end{equation}
where $\mathcal{F}:= \{(y,z,\theta) ~|~ (y,\theta) \in \mathrm{gph}(\Gamma)\}$. In particular, this tells us that
$$
\mathcal{N}_{\mathcal{F}}( y, z, \theta) = \{ (\xi_1, 0, \xi_3)~|~ (\xi_1,\xi_3) \in \mathcal{N}_{\mathrm{gph}(\Gamma)}(y,\theta)\}.
$$
It is easy to see that $\mathcal{N}_{\mathrm{gph}(\Gamma)}(y , \theta) \subset \mathbb{R}^{m_1} \times \mathcal{N}_{\Gamma(y)}(\theta)$, and consequently
\begin{equation*}
    \begin{aligned}
        &\{ \theta ~|~ 0 = F(\bar{y}, \bar{z}) + (\theta - \bar{z})/\gamma + \xi_3,  (\xi_1, \xi_2, \xi_3)\in \mathcal{N}_{\mathcal{F}}(\bar y, \bar{z}, \theta)\} \\
        =& \{ \theta ~|~ 0 = F(\bar{y}, \bar{z}) + (\theta - \bar{z})/\gamma + \mathcal{N}_{\Gamma(y)}(\theta)\} \\
        =& \{	\theta_\gamma^*(\bar y, \bar{z})\}.
    \end{aligned}
\end{equation*}
Therefore, combining the normal cone representation in \eqref{thm_proxreg_eq5}  at $(\bar{y}, \theta_\gamma^*(\bar y, \bar{z}))$ with \eqref{thm_proxreg_eq3} and \eqref{thm_proxreg_eq4}, we arrive at  the subdifferential formula 
\begin{equation*}
    \begin{aligned}
        \partial&\mu_\gamma(\bar y,\bar z)=\\
&\left\{ \left. \begin{bmatrix}
\nabla_y F(\bar{y}, \bar{z})^T \bar\theta^* \\
\nabla_z F(\bar{y}, \bar{z})^T \bar\theta^* 
\end{bmatrix}+  \begin{bmatrix}
\sum_{i = 1}^{s} \nu_i \nabla_y \zeta_i(\bar{y}, \bar\theta^*)  \\ (\bar{z} - \bar\theta^*)/\gamma 	
\end{bmatrix} 	~\right|~ 
\begin{array}{c}
\nu \in \mathbb{R}^s, \nu \ge 0, \sum_{i = 1}^{s}\nu_i \zeta_i(\bar{y}, \bar\theta^* ) = 0 \\
F(\bar y,\bar z)+(\bar\theta^*-\bar z)/{\gamma}+\sum_{i=1}^s\nu_i\nabla_z\zeta_i(\bar y,\bar\theta^*)=0
\end{array}   
\right\},
    \end{aligned}
\end{equation*}
where $\bar\theta^*:=\theta_\gamma^*(\bar y,\bar z)$ for the notational simplicity. This completes the proof of the proposition.
\end{proof}\vspace*{0.07in}

The preceding result demonstrates that the reformulation problem \eqref{VI-constraint_equivalent} constitutes a special case of the constrained difference programming problem \eqref{difference-programming}. This correspondence is established by setting $x: = (y,z)$, $X:= (Y \times \mathbb{R}^{m_2}) \cap \mathrm{gph}(\Gamma)$, and
\[
g_0(x):= \Psi_1(y,z), \quad h_0(x): = \Psi_2(y,z), \quad g_1(x): = \langle F(y,z),z\rangle,  \quad h_1(x): = \mu_\gamma(y,z).
\]  
Under Assumption~\ref{assumption1}(a,b), the reformulation \eqref{VI-constraint_equivalent} satisfies all the assumptions imposed on the constrained difference programming problem \eqref{difference-programming}. Consequently, the proposed Algorithm~\ref{alg:ESQM} (CDP-ESQM) can be directly applied to solve the reformulation problem \eqref{VI-constraint_equivalent}.

However, by employing the same proof technique as in \cite[Proposition~7]{ye2021difference}, it can be shown that NNAMCQ is {\em violated} at every feasible point of the reformulation problem \eqref{VI-constraint_equivalent}. This observation motivates the consideration of the following 
{\em $\epsilon$-relaxation problem} associated with \eqref{VI-constraint_equivalent}:
\begin{equation}\label{relaxation}
\begin{aligned}
\min_{ y \in Y, z\in \Gamma(y)}  ~~& 
\Psi(y,z)  \\
\text{subject to } ~~~~& \langle F(y,z),z\rangle - \mu_\gamma(y,z) \le \epsilon,
\end{aligned}
\end{equation}
where $\epsilon> 0$ is the relaxation parameter. The next proposition tells us that for any $\delta > 0$, there exists $\epsilon > 0$ such that we can find a local minimizer of the $\epsilon$-relaxation problem \eqref{relaxation}, which is $\delta$-close to the solution set of the reformulation problem \eqref{VI-constraint_equivalent}. The proof of this fact is a direct extension of that in \cite[Proposition~6]{ye2021difference}, and thus it is omitted.

\begin{proposition}\label{prop:prox}
Suppose that the solution set $\mathcal{S}^*$ of problem \eqref{VI-constrained} is nonempty and compact. Then for any $\delta > 0$ there exists $\bar{\epsilon} > 0$ such that for any $\epsilon \in (0,\bar{\epsilon}]$ we can find a pair $(y_{\epsilon},z_{\epsilon})$, which gives a local minimum to the $\epsilon$-relaxation problem \eqref{relaxation} and satisfies the condition $\mathrm{dist}((y_{\epsilon}, z_{\epsilon}),\mathcal{S}^*) < \delta$.
\end{proposition}

A significant advantage of the $\epsilon$-relaxation problem \eqref{relaxation} is that both NNAMCQ and ENNAMCQ are satisfied for it under certain conditions. We provide below a easily verifiable sufficient condition under which ENNAMCQ holds for the $\epsilon$-relaxation problem \eqref{relaxation}.

\begin{proposition} Let \( (\bar{y}, \bar{z}) \) satisfy $\bar{y} \in Y$ and $\bar{z} \in \Gamma(\bar{y})$. If the partial Jacobian matrix $\nabla_z F(\bar{y}, \bar{z})$ is positive-definite, then ENNAMCQ holds at \( (\bar{y}, \bar{z}) \) for the $\epsilon$-relaxation problem \eqref{relaxation} whenever $\epsilon >0$.
\end{proposition}
\begin{proof}\,
Denoting $\bar{\theta}^* =  \theta_\gamma^*(\bar{x}, \bar{y}) $ as above, suppose on the contrary that ENNAMCQ fails at $(\bar{y}, \bar{z})$. Then we have $ \langle F(\bar{y}, \bar{z}), \bar{z}\rangle - \mu_\gamma(\bar{y}, \bar{z}) \ge \epsilon$ and 
\[
0 \in \begin{bmatrix}
\nabla_y F(\bar{y}, \bar{z})^T\bar{z} \\ \nabla_z F(\bar{y}, \bar{z})^T\bar{z} + F(\bar{y}, \bar{z})
\end{bmatrix}
 - \partial \mu_\gamma(\bar{y}, \bar{z}) + \mathcal{N}_{\mathrm{gph}(\Gamma)}(\bar{y}, \bar{z}).
\]
It follows from \eqref{thm_proxreg_eq4} and the obvious inclusion $\mathcal{N}_{\mathrm{gph}(\Gamma)}(\bar{y}, \bar{z}) \subset \mathbb{R}^{m_1} \times \mathcal{N}_{\Gamma(\bar{y})}(\bar{z})$ that
\begin{equation}\label{prop_cq_eq1}
0 \in  \nabla_z F(\bar{y}, \bar{z})^T\bar{z} + F(\bar{y}, \bar{z}) -  \nabla_z F(\bar{y}, \bar{z})^T\bar{\theta}^* - (\bar{z} - \bar{\theta}^* )/\gamma + \mathcal{N}_{\Gamma(\bar{y})}(\bar{z}).
\end{equation}
By the standard first-order necessary optimality condition for the optimization problem defining $\bar{\theta}^*$ in \eqref{def_theta}, we get the inclusion
\begin{equation}\label{prop_cq_eq2}
0 \in F(\bar{y}, \bar{z}) + (\bar{\theta}^* - \bar{z})/\gamma + \mathcal{N}_{\Gamma(\bar{y})}( \bar{\theta}^*  ).
\end{equation}
Combining \eqref{prop_cq_eq1} and \eqref{prop_cq_eq2} and using the monotonicity of the normal cone mapping $\mathcal{N}_{\Gamma(\bar{y})}$ associated with the convex set $\Gamma(\bar y)$ give us the inequality
\begin{equation}\label{monot}
\left\langle \nabla_z F(\bar{y}, \bar{z})^T \left(\bar{z} - \bar{\theta}^*\right), \bar{z} -   \bar{\theta}^* \right\rangle \le 0.
\end{equation}
Since \( \nabla_z F(\bar{y}, \bar{z}) \) is positive-definite, \eqref{monot} yields $\bar{z} = \bar{\theta}^*$ and leads us to $\langle F(\bar{y}, \bar{z}), \bar{z}\rangle - \mu_\gamma(\bar{y}, \bar{z}) = 0$, which contradicts the assumption that $\langle F(\bar{y}, \bar{z}), \bar{z}\rangle - \mu_\gamma(\bar{y}, \bar{z}) \ge \epsilon$ and thus completes the proof.
\end{proof}

\section{Applications to Continuous Network Design with Numerical Experiments}\label{sec:num}
\setcounter{equation}{0}

This section tests the numerical performance of Algorithms~\ref{alg:ESQM} (CDP-ESQM) in the case of VI-constrained  difference programming. We apply {this algorithm} to the {\em continuous network design problem} (CNDP) in \cite{abdulaal1979continuous}, which is a model of practical value. The effectiveness of the new {algorithm} is demonstrated by numerical experiments by comparing {its} performance with the existing benchmarks for CNDP provided in \cite{suwansirikul1987edo}. We also compare {our algorithm} with those developed for mathematical programs with complementarity constraints (MPCCs) to evaluate their computational efficiency.  All the experiments were conducted in MATLAB on a laptop with Intel i7-1260 CPU (2.10 GHz) and 32 GB of RAM. Subproblems arising in the iterative process were solved by using MATLAB's built-in optimization toolboxes; namely, \texttt{quadprog} and \texttt{fmincon}, with the interior-point solver.

Following \cite{abdulaal1979continuous,marcotte1986network}, the CNDP model can be mathematically formulated as:
\begin{equation}\label{problem:CNDP}
\begin{aligned}
\min\limits_{y\geq0,\, v\in\Omega}\, &  \sum\limits_{a \in \mathcal{A}} t_a(y_a, v_a) v_a 
+G_a(y_a) \\
\text{subject to } &  v \in \mathcal{S}(y):=\left\{
v\in\Omega\,\middle|\, 
\langle t(y,v),v'-v\rangle\geq0,\ \forall v'\in\Omega
\right\},
\end{aligned}
\end{equation}
where \(\mathcal{A}\) is the set of directed links, 
\(y = (y_a)_{a \in \mathcal{A}}\) represents the vector of capacity expansion variables, 
and \(v = (v_a)_{a \in \mathcal{A}}\) denotes the vector of link flow variables. 
The travel time cost function is denoted by \(t_a(y_a, v_a)\), 
and \(G_a(y_a)\) represents the cost function associated with capacity expansion 
for each link \(a \in \mathcal{A}\). Moreover, the set $\Omega$ in \eqref{problem:CNDP} is defined by
\begin{equation}
\Omega:=\left\{v'=\Delta h\,\middle|\,\Lambda h=r,\,h\geq0\right\},
\end{equation}
where \(r\) is the vector of origin-destination demands. 

We evaluate the performance of CDP-ESQM on the network described in {\cite[Figure 3 and Table VI]{suwansirikul1987edo}}. Specifically, this network consists of 16 links implying that \( y, v \in \mathbb{R}^{16} \). For each \( a \in \mathcal{A} \), the cost functions are defined by
\begin{equation} t_a(y_a,v_a): = A_a + B_a \left(\frac{v_a}{K_a + y_a}\right)^4, \quad G_a(y_a) = D_a y_a,
\end{equation}
where the vectors \( A, B, K, D \) are taken from {\cite[Table VI]{suwansirikul1987edo}}. The matrices \( \Delta \) and \( \Lambda \) are obtained from {\cite[Figure~3]{suwansirikul1987edo}}, where \( \Delta \) is a \( 16 \times 16 \) binary matrix with elements 0 and 1 and \( \Lambda \) is given by
\begin{equation}
\Lambda: = 
\begin{bmatrix}
\mathbf{1}_8  & \mathbf{0}_8\\
\mathbf{0}_8  & \mathbf{1}_8
\end{bmatrix}^T.
\end{equation}

For CDP-ESQM in the network setting of {\cite[Figure~3 and Table~VI]{suwansirikul1987edo}}, the initial values were set as $y^0 = \mathbf{0}_{16}, v^0 = \frac{1}{8}\Delta(r \otimes \mathbf{1}_8 )$, and the common parameters were chosen as {$p_0=10, \varrho_p=0.1, c_p=5, \gamma=1, \epsilon =1, \beta =0.5,\sigma=0.5$}. {To investigate the algorithmic performance under different second-order approximations, we conducted two sets of experiments for each demand scenario by employing distinct strategies for the matrix $G_k$. The first variant, CDP-ESQM with a scaled identity matrix, employs  $G_k = \alpha \mathbf{I}$ with $\alpha = 10$. The second variant, CDP-ESQM with a regularized exact Hessian of the objective function, takes the form $G_k = \nabla^2 \varphi_0(y^k,v^k) + \alpha \mathbf{I}$.} The stopping criterion for both variants was defined as $\|d^k\|\leq {10^{-4}}$ with the constraint violation measure
\begin{equation}
\kappa_k=\left\langle t(y^k,v^k),v^k\right\rangle-\min\limits_{v\in\Omega}\left\{\left\langle t(y^k,v^k),v\right\rangle+\frac{1}{2\gamma}\left\|v-v^k\right\|^2\right\}\leq {\epsilon}.
\end{equation}

To evaluate the effectiveness of CDP-ESQM, we compare its numerical results with several existing benchmarks in the CNDP framework including the Modular In-Core Nonlinear Optimization System (MINOS), the Hooke-Jeeves (H-J) method, the Equilibrium Decomposed Optimization (EDO), and the Iterative Optimization-Assignment (IOA). The benchmark results are drawn from \cite[Table VII]{suwansirikul1987edo}, which also provides descriptions of these baseline methods.

In this application, the objective is to find an optimal value of the upper-level variable 
$y$ (the capacity expansion vector) such that, after solving the lower-level VI in CNDP for the link flow vector $v \in \mathcal{S}(y)$, the resulting pair $(y,v)$ yields a small value of the upper-level objective function (the total network cost).
By design, the benchmark algorithms (MINOS, H-J, EDO, and IOA) generate feasible pairs $(y,v)$ that satisfy this VI constraint $v \in \mathcal{S}(y)$, and their reported objective values are evaluated exactly at these feasible points.
To ensure a fair comparison, the final objective value for CDP-ESQM is computed using a two-step procedure. First, the capacity expansion vector $y$ obtained at the termination of CDP-ESQM is recorded. Then, the corresponding VI problem in the CNDP is solved to determine the link flow vector 
$v$. This gives a feasible pair $(y,v)$ for the CNDP. Finally, the total cost (i.e., the objective value) is calculated by substituting the obtained $y$ and $v$ into the upper-level objective function.

\begin{table}[htbp]
\centering
\caption{Objective values for different methods under three demand scenarios}
\label{tab:Comparison of Methods for Demand Scenarios}
\small 
\begin{tabular}{@{}lccc@{}}
\toprule
\textbf{Algorithm} & $r=(2.5,5.0)^T$ & $r=(5.0,10.0)^T$ & $r=(10.0,20.0)^T$ \\ 
\midrule
CDP-ESQM (Scaled Identity)   & {90.37}    & \bf{199.55}    &  \bf{528.24}    \\
CDP-ESQM (Regularized Hessian)  & {90.40}    & \bf{199.72}    & \bf{527.98}    \\
MINOS               & 92.10       & 211.25      & 557.14       \\
H-J                 & 90.10       & 215.08      & 557.22       \\
EDO                 & 92.41       & 201.84      & 540.74       \\ 
IOA                 & 100.25      & 210.86      & 556.61       \\ 
\bottomrule
\end{tabular}
\end{table}

Table~\ref{tab:Comparison of Methods for Demand Scenarios} summarizes the objective values computed under three different origin-destination demand scenarios \( r \), with the results presented in rows three to six sourced from \cite[Table VII]{suwansirikul1987edo}.  
From Table~\ref{tab:Comparison of Methods for Demand Scenarios}, it can be observed that in the mid- and high-demand scenarios, \( r=(5.0,10.0)^T \) and \( r=(10.0,20.0)^T \), the proposed algorithm achieves lower objective values with reducing them by 2.12 and {12.76}, respectively, compared to the best benchmark results.

To test the efficiency of CDP-ESQM, we compare its performance with the MPCC approach in optimization. The MPCC reformulation rewrites the constraints of CNDP as the following form of the KKT conditions:
\[
\Delta^T t(y,v) - \Lambda^T \mu \geq 0, \quad h \geq 0, \quad \langle \Delta^T t(y,v) - \Lambda^T \mu, h \rangle = 0, \quad \Delta h = v, \quad \Lambda h = r.
\] 
This reformulation is solved by using the MPCC approach with a relaxation parameter of 0.01 for the complementarity constraints.

We compare the objective values and feasibility as metrics to assess the degree of constraint violation during the iterative process. 
The objective values for both CDP-ESQM and the MPCC approach are calculated using the two-step procedure described above, consistent with the results presented in Table~\ref{tab:Comparison of Methods for Demand Scenarios}.
Feasibility at each iteration \( k \) is computed by using the following formula:
\[
\frac{\max\{0,\kappa_k\}}{\max\{1,\kappa_0\}}.
\]

We present the evolution of both the objective function values and feasibility over the iterations for the largest demand scenario. The results in 
Figures~\ref{fig:performance_combined}
highlight the efficiency of both CDP-ESQM variants (i.e., employing scaled identity and regularized Hessian matrices) compared to the MPCC approach. 
Both proposed variants reach the optimal objective value significantly faster and achieve tighter feasibility tolerances,
while exhibiting smoother and more stable convergence.

\begin{figure}[htbp]
\centering
   
\begin{subfigure}[b]{0.48\textwidth}
\centering
\includegraphics[width=\textwidth]{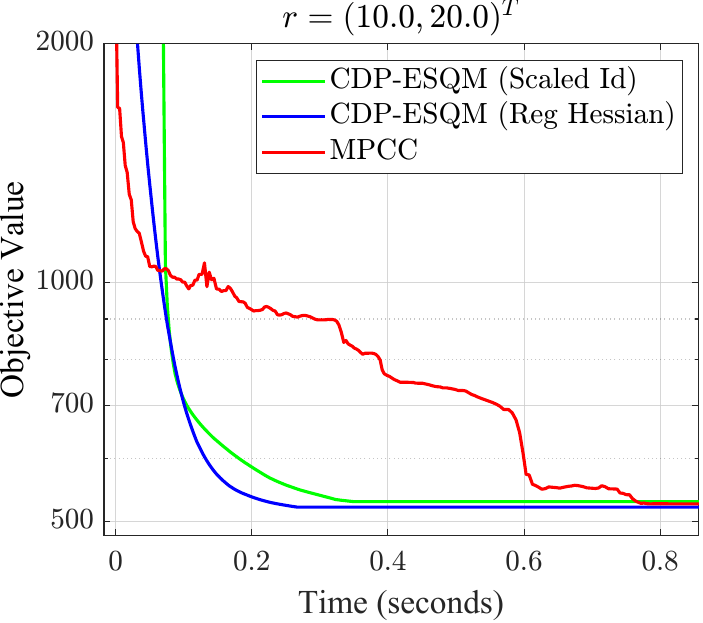}
\caption{Objective values for \( r=(10.0,20.0)^T \).}
\label{fig:obj_high}
\end{subfigure}
\hfill
\begin{subfigure}[b]{0.48\textwidth}
\centering
\includegraphics[width=\textwidth]{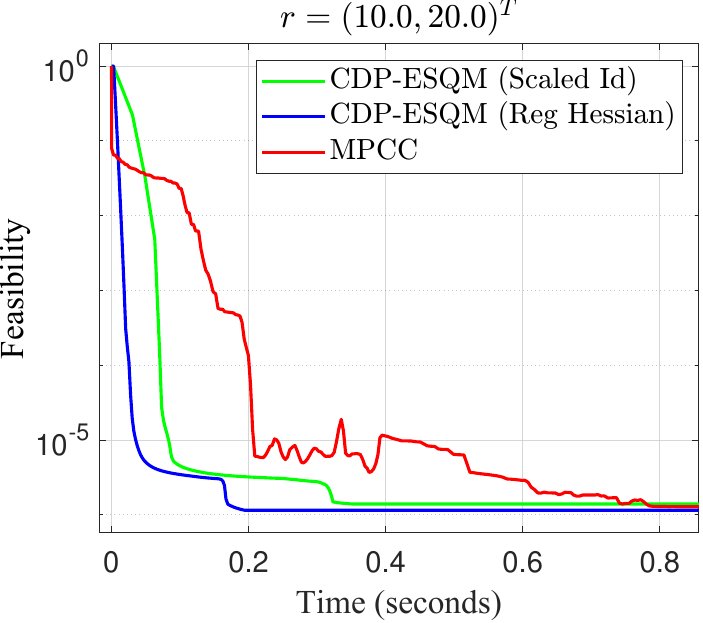}
\caption{Feasibility for \( r=(10.0,20.0)^T \).}
\label{fig:feas_high}
\end{subfigure}
    
\caption{Performance comparison of objective values and feasibility for demand scenario  \( r=(10.0,20.0)^T \).}
\label{fig:performance_combined}
\end{figure}

Finally, we analyze the behavior of the penalty parameter sequence \(\{p_k\}\) in CDP-ESQM. We do not set a stopping criterion.  The parameters 
remain unchanged. As illustrated by 
Figure~\ref{fig:penalty_stability} in the given demand scenario \(r = (10.0, 20.0)^T\), the penalty parameter \(p_k\) eventually converges to a constant value demonstrating in this way the stability of the proposed algorithms in the later iterations.

\begin{figure}[htbp]
    \centering
    \includegraphics[width=1.0\textwidth, height=4.5cm]{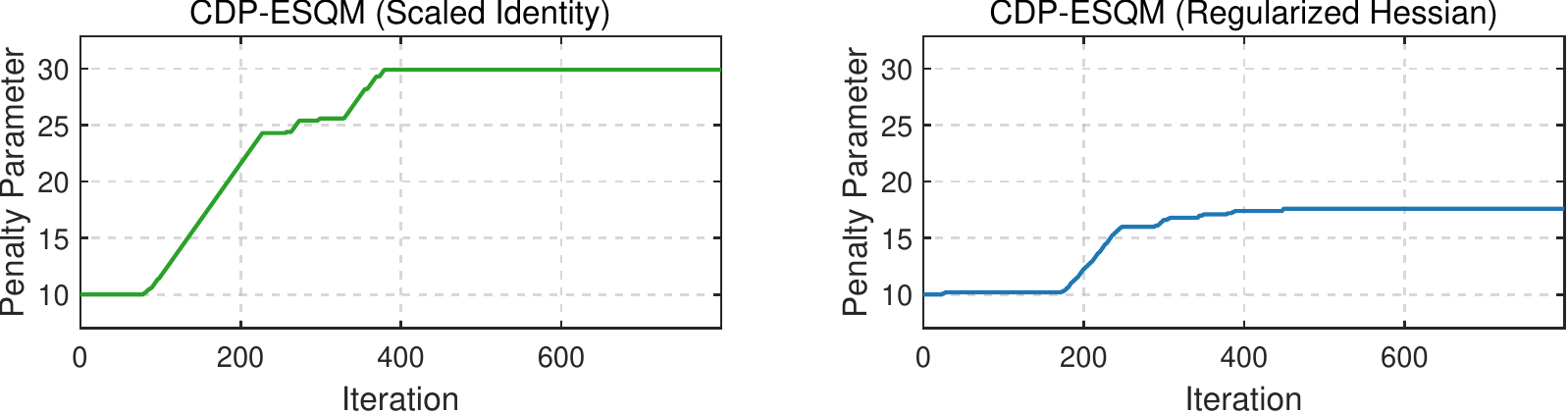}
    \caption{Evolution of the penalty parameter over iterations for demand scenario $r=(10.0, 20.0)^T$.}
    \label{fig:penalty_stability}
\end{figure}

\section{Concluding Remarks and Future Research}\label{conc}

This paper designs and justifies a novel algorithm of the generalized SQP type to solve a new class of constrained difference programming. The developed algorithm is implemented for difference programming problems with variational inequality constraints and then is applied to practical continuous network design models. For the latter class of models, numerical experiments were performed, and the computation results were compared with those obtained by using other numerical techniques.

Among directions of our future research, we mention investigating {\em convergence rates} for  CDP-ESQM. In particular, we plan to employ the exponential PLK conditions from \cite{bento2025convergence}
to determine exponent values ensuring the finite termination as well as linear and polynomial convergence of the designed algorithms. Another important area of the future research is considering problems of the constrained difference 
programming of type \eqref{difference-programming} (and the associated problems with variational inequality constraints) in which the functions $g_0$ and $g_i$, $i \in \mathcal{I}$ are {\em not of class ${\cal C}^{1,1}$} and even {\em not differentiable} on the set in question. Some developments in this direction have been recently done in \cite{ferreira} for the boosted difference of convex algorithm (BDCA) \cite{av20} in problems of DC programming.


\begin{thebibliography}{}

\bibitem{abdulaal1979continuous}
Abdulaal, M., LeBlanc, L.J.:
Continuous equilibrium network design models.
Transportation Research Part B: Methodological 13, 19--32 (1979)

\bibitem{thi2014dc}
An, L.T.H., Ngai, H.V., Tao, P.D.: 
DC programming and DCA for general DC programs. 
In: Advanced Computational Methods for Knowledge Engineering, pp. 15--35. 
Springer, Cham (2014)

\bibitem{le2005}
An, L.T.V., Tao, P.D.: 
The DC (difference of convex functions) programming and DCA revisited with DC models of
real-world nonconvex optimization problems. 
Annals of Operations Research 133, 23--46 (2005)

\bibitem{aragon2023coderivative}
Arag\'on-Artacho, F.J., Mordukhovich, B.S., P\'erez-Aros, P.:
Coderivative-based semi-Newton method in nonsmooth difference programming.
Mathematical Programming {213, 385-432 (2025)}

\bibitem{av20} Arag\'on-Artacho, F.J., Vuong, P.T.: The boosted difference of convex functions algorithm for nonsmooth functions. SIAM Journal on Optimization 30, 980--1006 (2020)

\bibitem{Attouch2013}
Attouch, H., Bolt\'e, J., Svaiter, B.F.:
Convergence of descent methods for semi-algebraic and tame problems: proximal algorithms, forward–backward splitting, and regularized Gauss–Seidel method.
Mathematical Programming 137, 91--129 (2013)

\bibitem{auslender2013extended}
Auslender, A.:
An extended sequential quadratically constrained quadratic programming algorithm for nonlinear,
semidefinite, and second-order cone programming.
Journal of Optimization Theory and Applications 
156, 183--212 (2013)

\bibitem{beck2017first}
Beck, A.:
First-Order Methods in Optimization. 
SIAM, Philadelphia, PA (2017)

\bibitem{bento2025convergence}
Bento, G.,  Mordukhovich, B.S.,  Mota, T., Nesterov, Y.:
Convergence of descent optimization algorithms under Polyak-Lojasiewicz-Kurdyka conditions. 
{Journal of Optimization Theory and Applications 207, 41 (2025)}

\bibitem{bertsekas2016nonlinear}
Bertsekas, D.P.: 
Nonlinear Programming, 3rd edition. 
Athena Scientific, Belmont, MA (2016)

\bibitem{boggs1995sequential}
Boggs, P.T., Tolle, J.W.: 
Sequential quadratic programming. 
Acta Numerica 4, 1--51 (1995)

\bibitem{Bolte2007}
Bolt\'e, J., Daniilidis, A., Lewis, A.S., Shiota, M.: 
Clarke subgradients of stratifiable functions. 
SIAM Journal on Optimization 18, 556--572 (2007)

\bibitem{Clarke1995}
Clarke, F.H., Stern, R.J., Wolenski, P.R.: 
Proximal smoothness and the lower-$C^2$ property. 
Journal of Convex Analysis 2, 117--144 (1995)

\bibitem{DuttaLafhimZemkohoZhou2025}
Dutta, J., Lafhim, L., Zemkoho, A., Zhou, S.:
Nonconvex quasi-variational inequalities: stability analysis and application to numerical optimization.
Journal of Optimization Theory and Applications 204, 16 (2025)

\bibitem{FacchineiPang2003}
Facchinei, F., Pang, J.S.: 
Finite-Dimensional Variational Inequalities and Complementarity Problems. 
Springer, New York (2003)

\bibitem{ferreira}  Ferreira, O.P., Mordukhovich, B.S., Santos, W.M.S., Souza, J.C.O.: 
An inexact boosted difference of convex algorithm for nondifferentiable functions. 
Journal of Optimization Theory and Applications 208, 71 (2026)


\bibitem{fukushima1992equivalent}
Fukushima, M.: 
Equivalent differentiable optimization problems and descent methods for asymmetric variational inequality problems. 
Mathematical Programming 53, 99--110 (1992)

\bibitem{gidel2019}
Gidel, G., Berard, H., Vignoud, G., Vincent, P., Lacoste-Julien, S.: 
A variational inequality perspective on generative adversarial networks. 
In: Proceedings of the International Conference on Learning Representations (ICLR), Arlington, WI (2019)

\bibitem{guo2015solving}
Guo, L., Lin, G.H., Ye, J.J.: 
Solving mathematical programs with equilibrium constraints. 
Journal of Optimization Theory and Applications 166, 234--256 (2015)



\bibitem{IzmailovSolodov2014}
Izmailov, A.F., Solodov, M.V.: 
Newton-Type Methods for Optimization and Variational Problems. 
Springer, Cham (2014)

\bibitem{kinderlehrer2000variational}
Kinderlehrer, D., Stampacchia, G.: 
An Introduction to Variational Inequalities and Their Applications. 
SIAM, Philadelphia, PA (2000)



\bibitem{karimi} H. Karimi, J. Nutini and M. Schmidt, Linear convergence of gradient and proximal-gradient methods under the Polyak-\L ojasiewicz condition, In: Machine Learning and Knowledge Discovery in Databases, Part~1, pp.\ 
795--811. Springer, Cham, Switzerland (2016)

\bibitem{kocvara1992nondifferentiable}
Ko\u cvara, M., Outrata, J.V.:  
A nondifferentiable approach to the solution of optimum design problems with variational inequalities. 
In: System Modelling and Optimization, pp. 364--373. Springer, Berlin, Heidelberg (1992)

\bibitem{lawrence2001}
Lawrence, C., Tits, A.: 
A computationally efficient feasible sequential quadratic programming algorithm. 
SIAM Journal on Optimization 11, 656--674 (2001)

\bibitem{panshaohua2024proximalnewton}
Liu, R., Pan, S., Wu, Y., Yang, X.:  
An inexact regularized proximal Newton method for nonconvex and nonsmooth optimization. 
Computational Optimization and Applications 88, 603–641 (2024)





\bibitem{lucet2001sensitivity}
Lucet, Y., Ye, J.J.: 
Sensitivity analysis of the value function for optimization problems with variational inequality
constraints. SIAM Journal on Control and Optimization 40, 699–--23 (2001)

\bibitem{marcotte1986network}
Marcotte, P.: 
Network design problem with congestion effects: a case of bilevel programming. 
Mathematical Programming 34, 142--162 (1986)

\bibitem{marcotte1996exact}
Marcotte, P., Zhu, D.L.: 
Exact and inexact penalty methods for the generalized bilevel programming problem. 
Mathematical Programming, 74: 141-157, 1996.

\bibitem{mordukhovich2006variational}
Mordukhovich, B.S.: 
Variational Analysis and Generalized Differentiation, I: Basic Theory, II: Applications. 
Springer, Berlin (2006)

\bibitem{mordukhovich2018variational}
Mordukhovich, B.S.: 
Variational Analysis and Applications. 
Springer, Cham, Switzerland (2018)

\bibitem{mordukhovich2024variational}
Mordukhovich, B.S.: 
Second-Order Variational Analysis in Optimization, Variational Stability, and Control: Theory, Algorithms, Applications.  Springer, Cham, Switzerland  (2024)

\bibitem{mordukhovich2022convex}
Mordukhovich, B.S., Nam, N.M.: 
Convex Analysis and Beyond, Volume I: Basic Theory. 
Springer, Cham, Switzerland (2022)

\bibitem{mordukhovich2023easy}
Mordukhovich, B.S., Nam, N.M.: 
An Easy Path to Convex Analysis and Applications, 2nd edition, 
Springer, Cham, Switzerland (2023)

\bibitem{nagurney1993network}
Nagurney, A.: 
Network Economics: A Variational Inequality Approach. 
Kluwer Academic Publishers, Dordrecht, Netherland (1993)

\bibitem{nocedal2006numerical}
Nocedal, J., Wright, S.J.: 
Numerical Optimization, 2nd edition. 
Springer, New York (2006)

\bibitem{outrata} 
Outrata, J.V., Ko\u cvara, M., Zowe, J.:
Nonsmooth Approach to Optimization Problems with Equilibrium Constraints. 
Kluwer Academic Publishers, Dordrecht, Netherland (1998)



\bibitem{pang2017computing}
Pang, J.S., Razaviyayn, M., Alvarado, A.: 
Computing B-stationary points of nonsmooth DC programs. 
Mathematics of Operations Research 42, 95--118 (2017)

\bibitem{rockafellar1998variational}
Rockafellar, R.T., Wets, R.J-B.: 
Variational Analysis.  Springer, Berlin (1998)

\bibitem{samadi2025improved}
Samadi, S., Yousefian, F.: 
Improved guarantees for optimal Nash equilibrium seeking and bilevel variational inequalities. 
SIAM Journal on Optimization 35, 369--399 (2025)

\bibitem{scholtes2001convergence}
Scholtes, S.: 
Convergence properties of a regularization scheme for mathematical programs with complementarity
constraints. 
SIAM Journal on Optimization 11(4), 918--936 (2001)

\bibitem{steffensen2010new}
Steffensen, S., Ulbrich, M.: 
A new relaxation scheme for mathematical programs with equilibrium constraints.
SIAM Journal on Optimization 20, 2504--2539 (2010)

\bibitem{suwansirikul1987edo}
Suwansirikul, C., Friesz, T.L., Tobin, R.L.: 
Equilibrium decomposed optimization: a heuristic for the continuous equilibrium network design problem. 
Transportation Science 21, 254--263 (1987)

\bibitem{tao1997}
Tao, P.D.,  An, L.T.H.: 
Convex analysis approach to DC programming: theory, algorithms, and applications. 
Acta Mathematica Vietnamica 22, 289–-355 (1997)

\bibitem{xu2015smoothing}
Xu, M., Ye, J.J., Zhang, L.:
Smoothing SQP methods for solving degenerate nonsmooth constrained optimization problems with applications to bilevel programs.
SIAM Journal on Optimization 25, 1388--1410 (2015)


\bibitem{ye2000constraint}
Ye, J.J.: 
Constraint qualifications and necessary optimality conditions for optimization problems with variational inequality constraints. 
SIAM Journal on Optimization 10, 943--962 (2000)

\bibitem{ye2021difference}
Ye, J.J., Yuan, X., Zeng, S., Zhang, J.: 
Difference of convex algorithms for bilevel programs with applications in hyperparameter selection. 
Mathematical Programming 198, 1583--1616 (2023)


\end{thebibliography}
\end{document}